\documentclass[reqno, 12pt]{amsart}

\usepackage[utf8]{inputenc}
\usepackage{lmodern}
\usepackage[english]{babel}
\usepackage{csquotes}
\usepackage{amstext}
\usepackage{amsthm}
\usepackage{amsmath}
\usepackage{amssymb}
\usepackage{mathtools}
\usepackage[svgnames]{xcolor}
\usepackage[margin=1in]{geometry}
\usepackage{dsfont}
\usepackage{fancyvrb}
\usepackage{enumitem}
\usepackage{tikz}

\usepackage[
backend=biber,
style=numeric,
giveninits=true,
doi=false,
isbn=false,
url=false,
eprint=false,
maxnames=99
]{biblatex}
\DeclareListFormat{language}{}

\renewbibmacro{in:}{
    \ifentrytype{article}{}{\printtext{\bibstring{in}\intitlepunct}}}

\usepackage[unicode=true,
    pdfusetitle,
    bookmarks=true,
    bookmarksnumbered=false,
    bookmarksopen=false,
    breaklinks=false,
    pdfborder={0 0 1},
    backref=false,
    colorlinks=true]{hyperref}

\hypersetup{
    pdftitle={Defocusing NLS on graphs},
    pdfauthor={Élio Durand-Simonnet,
    Damien Galant and Boris Shakarov},
    pdfsubject={},
    pdfkeywords={},
    colorlinks=true,
    linkcolor=DarkBlue,
    citecolor=DarkRed,
    filecolor=DarkMagenta,
    urlcolor=DarkGreen}

\newtheorem{theorem}{Theorem}[section]
\newtheorem{lemma}[theorem]{Lemma}
\newtheorem{proposition}[theorem]{Proposition}
\newtheorem{corollary}[theorem]{Corollary}
\newtheorem{remark}[theorem]{Remark}
\theoremstyle{definition}
\newtheorem{definition}[theorem]{Definition}

\def\R{\mathbb R}
\def\C{\mathbb C}

\def\Z{\mathbb Z}

\def\G{\mathcal G}
\def\H{\operatorname{H}}
\def\DQH{\mathcal{D}(Q_{\H})}
\def\eps{\epsilon}
\def\Llow{\underline{L}}
\newcommand\MM{{\mathcal{M}}}

\DeclareMathOperator{\lspan}{span}

\DeclareMathOperator{\supp}{supp}

\DeclarePairedDelimiterX{\dual}[2]{\langle}{\rangle}{#1, #2}

\makeatletter
\def\munderbar#1{\underline{\sbox\tw@{$#1$}\dp\tw@\z@\box\tw@}}
\makeatother

\title[Defocusing NLS on graphs]{On the defocusing stationary nonlinear Schrödinger equation on metric graphs}

\author[E. Durand-Simonnet]{Élio Durand-Simonnet}
\address{Institut de Mathématiques de Toulouse ;
    UMR5219, Université de Toulouse ;
    CNRS, UPS IMT, F-31062 Toulouse Cedex 9 (France)}
\email{elio.durand\_simonnet@math.univ-toulouse.fr}

\author[D. Galant]{Damien Galant}
\address{(Primary) Department of Mathematics;
Brown University;
Providence, RI 02912 (USA)}
\address{(Secondary) Département de Mathématique;
UMONS, Université de Mons; 7000 Mons (Belgium)}
\email{damien\_galant@brown.edu}

\author[B. Shakarov]{Boris Shakarov}
\address{Institut de Mathématiques de Toulouse ;
    UMR5219, Université de Toulouse ;
    CNRS, UPS IMT, F-31062 Toulouse Cedex 9 (France)}
\email{boris.shakarov@math.univ-toulouse.fr}

\thanks{The work of E.\,D.S. and B.\,S. is partially supported by the ANR project CIMI ANR-11-LABX-0040
and the ANR project NQG ANR-23-CE40-0005.
E.\,D.S. is partially supported by ANR project MINT ANR-18-EURE-0023.
D.\,G. is a Francqui Fellow of the Belgian American Educational Foundation at Brown University
and was a Research Fellow of the F.R.S.-FNRS when this research was initiated.
He would like to thank the Francqui Foundation, the Belgian American Educational Foundation,
the F.R.S.-FNRS and the ANR project NQG ANR-23-CE40-0005 for supporting this research.}

\date{\today}

\subjclass[2020]{35Q55 (35A15, 35B38, 37K45, 35R02)}

\keywords{Nonlinear Schrödinger equation, metric graphs, energy ground states, bifurcation, Lusternik–Schnirelman theory}

\begin{document}

\begin{abstract}
    We study the
    defocusing nonlinear Schr\"odinger equation
    on noncompact metric graphs under general
    self-adjoint vertex conditions
    ensuring the existence
    of a negative eigenvalue of the Hamiltonian operator.
    First, we focus on the existence of energy ground states 
    with prescribed mass. We show that existence and stability
    always hold for small masses and
    fail for large masses
    in the $L^2$-subcritical regime.
    For $\delta$-type vertex conditions, we provide more precise results: ground states exist for all masses in the $L^2$-critical and supercritical cases, while in the subcritical case, for one vertex graphs, there exists a sharp mass threshold such that ground states exist below it and do not exist above it.
    Moreover, we show that the ground state bifurcates from the vanishing solution at the bottom of the Hamiltonian spectrum.
    Finally, we present multiplicity results for stationary solutions,
    both in the fixed-frequency and fixed-mass settings.
\end{abstract}

\maketitle

\section{Introduction}

The study of stationary solutions to nonlinear Schrödinger equations on metric graphs has been a very active area of research in recent decades as highlighted, for instance, in the review \cite{KaNoPe22}. This interest originates from the fact that graphs provide an effective one-dimensional approximation of wave propagation in network-like domains with small transverse dimensions \cite{Ko00, CoSh24, CoSh25Bk}.

Nevertheless, the majority of existing results
treat the case of a focusing nonlinearity.
Much less is known about the case
of a \textit{defocusing} nonlinearity.
In the present work, we pursue a study initiated
in \cite{DuSh26} and consider solutions
of the \textit{defocusing stationary nonlinear Schrödinger equation}
\begin{equation} \label{eqNls}
    \H \psi + \omega \psi + |\psi|^{p-1} \psi = 0.
\end{equation}

Here, the \textit{Hamiltonian operator} $\H$
acts on each edge of the graph $\G$ as the usual Laplacian in space,
endowed with proper boundary conditions on the vertices
(see Section \ref{secPrel} or \cite[Chapter 1]{BeKu13}
for a comprehensive introduction).
The solution $\psi$ in \eqref{eqNls}
is a collection $\psi = (\psi_e)_{e \in  \mathcal E}$
of solutions of the stationary equation 
\begin{equation} \label{eqNls2}
      \omega \psi_e + |\psi_e|^{p-1} \psi_e = \psi_e''
\end{equation}
set on each edge separately, with appropriate boundary conditions on the vertices,
as prescribed by the definition of $\H$.
The quadratic form associated with $\H$ is given by

\begin{equation}\label{eqQuadForm}
    Q_{\H} (\psi) \coloneqq \left\langle \H \psi, \psi \right\rangle_{L^2(\mathcal G)}
\end{equation}
for any $\psi \in \DQH\subseteq H^1(\G)$
and its domain $\DQH$ depends on a choice
of vertex conditions, see \eqref{eqDomainQ} below. Solutions to \eqref{eqNls} are stationary states of the time-dependent nonlinear Schrödinger equation
(see \cite[Proposition 2.4]{DuSh26} for a well-posedness
statement)
\begin{equation} \label{eqTimeDepNLS}
    \begin{cases}
        i \partial_t u = \operatorname{H} u + |u|^{p-1} u, \\
        u(t = 0) = \psi \in \DQH,
    \end{cases}
\end{equation}
for \( u : \mathbb{R} \times \mathcal{G} \to \mathbb{C} \) via the ansatz $u(t,x) = e^{-i\omega t} \psi(x)$. This time dependent equation admits two conserved quantities, which are  the $L^2$ norm of the solution called the \textit{mass} and the \textit{energy} given by
\begin{equation}\label{eqEnIntr}
    E_{\H}(\psi) \coloneqq \frac{1}{2} Q_{\H}(\psi) + \frac{1}{p+1} \| \psi \|_{L^{p+1}(\G)}^{p+1}.
\end{equation}
We consider the following minimization problem:
given any $\mu > 0$, let
\begin{equation} \label{eqEnMin}
    \tau_\mu \coloneqq \inf \left\{ E_{\H}(\psi) \mid \psi \in \DQH ,\, \| \psi \|_{L^2(\G)} = \mu \right\}.
\end{equation}

Every minimizer to \eqref{eqEnMin} satisfies \eqref{eqNls}
with a Lagrange multiplier $\omega \in \R$
and is thus a \textit{normalized solution of mass $\mu$}.
We define the set of all minimizers as
\begin{equation} \label{eqSolSet}
    \mathcal B_\mu \coloneqq \left\{ \psi \in \DQH \mid \| \psi \|_{L^2(\G)} = \mu, \, E_{\H} (\psi) = \tau_\mu \right\}.
\end{equation}
We call elements of $\mathcal B_\mu$
\textit{energy ground states (of mass $\mu$).}
Our first aim is to describe in details the existence
of energy ground states, continuing a study initiated
in \cite{DuSh26}.

As we will see, the bottom of the spectrum of $\H$
will play an important role in our analysis.
More precisely, we define
\begin{equation} \label{eqLGamOm}
    l_{\H} \coloneqq \inf \left\{ Q_{\H}(\psi) \mid \psi \in \DQH, \, \| \psi \|_{L^2(\G)} = 1 \right\}.
\end{equation}
Our first result shows that energy ground states always exist
for small masses independently of the vertex conditions,
as long as $l_{\H} < 0$, that is,
as long as the vertex conditions introduce
a negative contribution to the energy in \eqref{eqEnIntr}.
Moreover, we show that in the $L^2$-subcritical
case $p\in (1,5)$, minimizers do not exist for large masses.
Furthermore, as a consequence of the conservation of mass
and energy of \eqref{eqTimeDepNLS}, we will also prove
the stability of the set $\mathcal B_\mu$
in the following sense. 

\begin{definition}\label{defStab}
    A set $\mathcal A$ is said to be \textit{stable} if for any $\varepsilon > 0$, there exists $\eta > 0$ such that if an initial condition $u_0 \in D(Q_{\H})$ satisfies 
    $$ 
    \inf_{q \in \mathcal{A}} \| u_0 - q \|_{H^1(\mathcal G)} < \eta,
    $$ 
    then the corresponding solution $u$ to \eqref{eqTimeDepNLS} satisfies
\[
    \sup_{t \geq 0} \inf_{q\in \mathcal{A}} \| u(t, \cdot) - q \|_{H^1(\mathcal G)} < \varepsilon.
\] 
\end{definition}

\begin{theorem}\label{thmPSmall}
    Let $\G$ be noncompact.
    Assume that $l_{\H}<0$.
    Then, there exists $\mu_1 >0$ such that $\tau_\mu$
    admits a minimizer for any $\mu \in [0,\mu_1]$.
    Moreover, if $p\in(1,5)$, then there exists
    $\mu_2 \geq \mu_1$ such that $\tau_\mu$
    does not admit any minimizer for $\mu > \mu_2$.
    Finally, if $B_\mu \neq \emptyset$ for some $\mu > 0$,
    then it is stable as in Definition \ref{defStab}.
\end{theorem}
The nonexistence of ground states for large masses is closely related to the existence of a global, \textit{unconstrained} minimizer of the energy. More precisely, in Proposition~\ref{prpVertCont} below, we establish a slightly stronger statement: whenever a global minimizer exists, the nonexistence of ground states for large masses follows. The conclusion of Theorem~\ref{thmPSmall} then relies on the fact that for $p \in (1,5)$ a global minimizer always exists, whereas for $p \geq 5$ the situation remains unclear when the compact core of the graph is nontrivial, see Proposition \ref{prpOmegaZero} below.

The question of (non)existence of energy ground states for $\mu_1 < \mu < \mu_2$ is more subtle due to the generality of the vertex conditions. For instance, the general nature of the vertex conditions
makes it difficult to establish the positivity of ground states.
As we shall see, closely related to this issue,
the sign of the frequency $\omega$ associated with energy ground states also plays a crucial role in our analysis (see, e.g., Lemma~\ref{lemPosMult}). 
Related to this, the special case of $\delta$-type conditions at vertices is worth considering specifically
(see Section~\ref{sec:delta} for definitions).
This type of vertex condition has received a lot of attention in the study of the focusing nonlinear Schrödinger equation on graphs,
as it models \textit{point defects}
\cite{AdCaFiNo14,AdCaFiNo16En2, AlBrD95,AlbGeHo05}.
In this case, we show the following
refined version of Theorem~\ref{thmPSmall}.
\begin{theorem}\label{thmDelta}
    Let $\G$ be noncompact.
    Assume $\H$ is defined with only $\delta$-type conditions
    at vertices and that $l_{\H}<0$.
    Then, any minimizer to $\tau_\mu$
    is positive up to a phase shift.
    Moreover, the following holds: 
    \begin{enumerate}
        \item If $p \geq 5$, then
        for any $\mu \in (0,\infty)$, $\tau_\mu$
        admits a minimizer.
        \item If $p \in(1, 5)$ and $\G$
        has only one vertex,
        then there exists $\mu_1 >0$
        such that $\tau_\mu$ admits a minimizer
        for any $\mu \in [0,\mu_1]$ and does not admit
        any minimizers for $\mu > \mu_1$. 
    \end{enumerate}
    Finally, for $\mu >0$ sufficiently small,
    the ground state bifurcates from the linear one
    (in the sense of Proposition \ref{prpBifurcation}).
\end{theorem}
The key improvement that allows us to obtain the stronger result of Theorem~\ref{thmDelta} is the continuity at the vertices. This property can be exploited to establish positivity of ground states (up to a complex phase shift). Combined with the nonexistence of nontrivial $H^1$ solutions on the half-lines when $p \geq 5$, this yields the first statement of Theorem~\ref{thmDelta}. The second statement follows instead from the uniqueness of solutions for a fixed frequency $\omega \geq 0$ on graphs with a single vertex, see Lemma~\ref{lemMinMax} below. Finally, positivity is also exploited to establish a Lyapunov-Schmidt reduction, showing that a branch of stationary states bifurcates from the smallest eigenvalue of $\H$ in the direction of the linear ground state. 
 
In general, as it has been noticed in \cite[Remark $3.11$]{DuSh26}, the question of (non)existence of energy ground states for intermediate masses is related to the uniqueness of the solution to \eqref{eqNls}. This is one motivation for the second part of this work. 
 
Here we will also consider a second functional whose critical points are solutions to \eqref{eqNls}. Instead of prescribing the masses of solutions, one may rather prescribe their frequency~$\omega$. More precisely, given $\omega \in \R$,
the \textit{action functional}
$S_{\omega}: \DQH \to \R$ is defined by
\begin{equation}\label{eqAcIntr}
    S_{\omega}(\psi)
    \coloneqq E_{\H}(\psi) + \frac{\omega}{2} \| \psi \|_{L^2(\G)}^2.
\end{equation}
We note that the solutions in $\DQH$
of \eqref{eqNls} correspond exactly to critical points
of $S_{\omega}$.
Moreover, nonzero global minimizers of $S_{\omega}$ are called
\textit{action ground states}.
In \cite{DuSh26}, it was shown that action ground states
exist if and only if $l_H < 0$ and $\omega \in [0, -l_H)$.

Considering the aforementioned result
alongside Theorem~\ref{thmPSmall} gives an extensive existence
theory for energy/action ground states.
Let us now turn our attention to multiplicity results,
either for solutions with a fixed frequency $\omega$
or for normalized solutions with a fixed mass $\mu$.

In analogy with the condition $l_{\H} < 0$ used to ensure the existence of ground states, we will require the presence of multiple negative eigenvalues\footnote{As will be shown
in Section~\ref{secPrel}, all negative elements of the spectrum of $\H$ are eigenvalues.}
in order to derive our multiplicity results.
 
As already pointed out in \cite{DuSh26}, due to the focusing nature of the nonlinearity, there are no compactness issues for the action
functional, which satisfies the Palais-Smale condition
(see Proposition~\ref{PSaction}).
We thus obtain the following theorem
by applying Lusternik-Schnirelmann theory to $S_{\omega}$.

\begin{theorem}\label{thmMultAction}
    Let $\G$ be noncompact.
    Assume that $\H$ has $k$ negative eigenvalues
    $\lambda_1 \le \cdots \le \lambda_k < 0$
    (eigenvalues having multiplicity being repeated).
    Then, for all $\omega \in (0, -\lambda_k)$,
    the problem \eqref{eqNls} admits at least $k$
    distinct $S^1$-orbits of solutions
    (where the orbit of $u$
    is $\{e^{i\theta}u \mid \theta \in [0, 2\pi)\}$)
    in $\DQH$, all having a negative action level.
    
    Furthermore, if the matrices $\Lambda_v$ defined in Proposition \ref{prpHSelfAdj}
    are real symmetric for all $v \in \mathcal{V}$, then these solutions can be chosen to be real-valued,
    yielding at least $k$ distinct pairs
    $\pm \psi_1, \cdots, \pm \psi_k$.
\end{theorem}

Finally, we obtain a multiplicity result
for fixed mass solutions in the small mass regime.
\begin{theorem}\label{thmMultNormalized}
    Let $\G$ be noncompact.
    Assume that $\H$ has $k$ negative eigenvalues
    $\lambda_1 \le \cdots \le \lambda_k < 0$
    (eigenvalues having multiplicity being repeated).
    Then, there exists $\tilde{\mu} > 0$
    such that, for all $\mu \in (0, \tilde{\mu})$,
    problem \eqref{eqNls} admits at least $k$
    distinct $S^1$-orbits of solutions in $\DQH$,
    all having negative energy and mass $\mu$.
    
    Furthermore, if the matrices $\Lambda_v$ defined in Proposition \ref{prpHSelfAdj}
    are real symmetric for all $v \in \mathcal{V}$, then these solutions can be chosen to be real-valued,
    yielding at least $k$ distinct pairs
    $\pm \psi_1, \cdots, \pm \psi_k$.
\end{theorem}

This time, the compactness is more delicate
and requires an analysis of the sign of the Lagrange multipliers
corresponding to Palais-Smale sequences,
as was the case for the existence theory of energy ground states,  see Proposition~\ref{propPSsignMultipliers} below.
One can then obtain our multiplicity result
using Lusternik-Schnirelmann theory on the energy functional
constrained to the manifold of functions having mass $\mu$.

To conclude the introduction, let us compare our results with the existing literature. As mentioned earlier, the literature for the focusing case (with the opposite sign in front of the nonlinearity in \eqref{eqNls}) is quite vast, and we do not aim to cover it here. However, the comparison between action and energy ground states is particularly interesting. Indeed, in the focusing case, energy ground states are action ground states (see \cite[Theorem~1.3]{dovetta_serra_tilli_2023}), while the converse may not be true \cite{DDGS,dovetta_serra_tilli_2023}. In our defocusing setting, by contrast, all action ground states are energy ground states, see \cite[Theorem~1.3]{DuSh26}, while the converse may fail, see \cite[Theorem~1.4]{DuSh26}.

Moreover, the inability to reach all masses in the energy minimization has different origins in the focusing and defocusing cases. In the focusing case, due to the convexity of the energy, a minimizing sequence may lose compactness primarily because of the \textit{run-away} scenario in the concentration-compactness theorem adapted to graphs \cite[Section~3]{CaFiNo17}; that is, when a line soliton is energetically favored over a state concentrated on the core
of the graph (see e.g.\,\cite{AdCaFiNo14}). In our case, this situation does not occur since there are no nontrivial line solitons; instead, the loss of compactness of the minimizing sequence can only occur via the \textit{dichotomy} scenario.

Finally, we note that the situation considered here is reminiscent of the case of competing nonlinearities \cite{AdBoDo22Nl, BoDo22Nl, SoVi26GSE}
or of nonlinear point interactions \cite{BaBoDoTe25NPI},
in the sense that some attractive phenomenon counterbalances the repulsive effect of one nonlinearity. A similar scenario can also arise in the presence of a negative linear potential; see, e.g., \cite{AdGaSp25Pot, CaFiNo17}.

The paper is organized as follows.
In Section~\ref{secPrel}, we precise our functional
setup and present several useful inequalities and lemmas.
In Section~\ref{secEGS}, we develop our existence theory
for energy ground states of a given mass.
We begin with the general setting,
proving Theorem~\ref{thmPSmall}, then focus on the special case
of $\delta$-type conditions at vertices.
 Then in the last part of this section, we show our bifurcation result, concluding the proof of Theorem~\ref{thmDelta}. 
Finally, Section~\ref{secMult} is devoted to our multiplicity
results using Lusternik-Schnirelmann theory. We begin with the
action case (Theorem~\ref{thmMultAction}),
and then prove the more delicate
energy case (Theorem~\ref{thmMultNormalized})
using a dedicated compactness argument
(Proposition~\ref{propPSsignMultipliers}).

\section{Preliminaries}\label{secPrel}

\subsection{Metric graphs}

In this section, we recall the basic theory of Laplace operators on a metric graph $\G$.
A metric graph is a $3$-tuple $\G = (\mathcal V, \mathcal{E}_{\operatorname{fin}}, \mathcal{E}_{\operatorname{inf}})$, where $V$ is the set of vertices, $\mathcal{E}_{\operatorname{fin}}$ the set of finite length edges, and $\mathcal{E}_{\operatorname{inf}}$ the set of infinite length edges.
Elements of $\mathcal{E} = \mathcal{E}_{\operatorname{fin}} \cup \mathcal{E}_{\operatorname{inf}}$ will be referred to as edges.

Throughout the paper, we assume that $\G$ is connected, meaning that for any two vertices 
$v, v' \in \mathcal V$ there exists a finite sequence $(v_1 = v, v_2, \dots, v_n = v')$ 
such that each pair of consecutive vertices is adjacent. 
For $v \in \mathcal V$, we denote by $J_v$ the set of edges adjacent to $v$, meaning that the corresponding edge can be identified with an interval starting from $v$. We denote by $d_v$ the cardinality or \textit{degree} of the set $J_v$. As $\G$ is connected, every vertex has a nonzero degree.
Throughout the paper, the sets $\mathcal E$ and $\mathcal V$ are assumed to be of finite cardinality. A graph is said to be compact if $\mathcal{E}_{\operatorname{inf}} = \emptyset$.  
An edge $e \in \mathcal E$ has a length $L_e$ that can be finite or infinite. The coordinate system on the metric graph is largely arbitrary. When the length $L_e$ is finite,
an edge $e \in \mathcal E$ can be identified with an interval $I_e \subset \R$ of length $L_e$, such as $I_e = [0, L_e]$.
When the length $L_e$ is infinite, the edge can be identified with any half-line $I_e = [a,\infty)$,
where $a \in \R$ can be chosen freely. Let $\Llow$ be defined as
\begin{equation} \label{eqMinLnght}
    \Llow \coloneqq \frac{1}{2}
    \min \left( \min_{e \in \mathcal E} L_e, 1 \right),
\end{equation}
which is well defined and positive since $|\mathcal E| < \infty$.

A point $\tilde{x} \in \G$ on the graph can be identified by giving the edge $e \in \mathcal{E}$ and the coordinate $x$
on the edge: $\tilde{x} = (e, x)$.
With a slight abuse of notation, we will use $x$
for both a point on the graph and a coordinate on one edge when there is no ambiguity.

A function on the metric graph is defined component-wise, giving its value on any edge.
Assigning  $\psi$ on $\G$  corresponds to specifying its edge components  $\{\psi_e\}_{e \in \mathcal{E}}$,
where each $\psi_e$  is a function $\psi_e: I_e \to \C$.
Equivalently, for a point $\tilde x = (e, x)$ on the graph,
we define  $\psi(\tilde x) = \psi_e(x)$.

For a noncompact graph, we define
the \textit{compact core}\footnote{Note that our definition differs slightly
from the usual one in the literature due to the presence of the edges $K_e$:
here we also truncate half-lines.} $\mathcal{K}$ as follows:
for any $e \in \mathcal{E}_{\operatorname{inf}}$,
identified with $[0,\infty)$, let $K_e = [0, \Llow] \subset e$, and let
\begin{equation} \label{eqCompCore}
    \mathcal{K} = \bigcup_{e \in \mathcal{E}_{\operatorname{fin}}} e \cup \bigcup_{e\in \mathcal{E}_{\operatorname{inf}}} K_e.
\end{equation}

We define
\[
    \int_{\G} \psi (x) dx \coloneqq \sum_{e \in \mathcal E} \int_{I_e} \psi_e (x) dx,
\]
and the functional spaces\footnote{Contrary to some papers,
here we do not impose any ``vertex conditions'' for functions in $H^k(\G)$, the regularity is just
measured edge-wise. For instance, with our convention, $H^1(\G)$ functions are not necessarily continuous.}
\begin{equation*}
    L^p (\G) = \bigoplus_{e \in \mathcal E} L^p(I_e), \quad H^k (\G) = \bigoplus_{e \in \mathcal E} H^k(I_e), 
\end{equation*}
endowed with the norms 
\begin{equation*}
    \left\| \psi \right\|_{L^p (\mathcal G)}^p = \sum_{e \in \mathcal E} \left\| \psi_e \right\|_{L^p(I_e)}^p, \quad \left\| \psi \right\|_{H^k (\mathcal G)}^2 = \sum_{e \in \mathcal E} \left\| \psi_e \right\|_{H^k(I_e)}^2
\end{equation*}
 and the scalar product on $L^2(\G)$ as
\[
    \langle \psi, \phi \rangle_{L^2 (\mathcal G)}
    = \operatorname{Re}\left( \int_{\G} \psi (x) \overline{\phi(x)} dx \right)
    = \sum_{e \in \mathcal E} \operatorname{Re}
    \left( \int_{I_e} \psi_e(x) \overline{\phi_e(x)} dx \right).
\]
For $\psi \in H^1(\mathcal G), v \in \mathcal V$ and $e \in J_v$,
we denote by $\psi_e(v)$ the value of $\psi_e$ at the vertex $v$
and by $\Psi(v) \in \C^{d_v}$ the vector given by
\[
    \Psi(v) \coloneqq (\psi_{e_1} (v), \ldots, \psi_{e_{d_v}} (v)),
\]
where $\{ e_1, \dots, e_{d_v} \} = J_v$. 

\subsection{Hamiltonian operator}

We equip the metric graph with the operator $\H : L^2(\mathcal G) \to L^2(\mathcal G)$ which is a self-adjoint extension of the one-dimensional Laplacian. We follow the construction given in \cite[Chapter 1]{BeKu13}, and we report just briefly the main characteristics. As $\H$ is an extension of the Laplacian,  then in any interior part of any edge, we have 
\[
    \H \psi = (-\psi_e'')_{e \in \mathcal E}
\]
and this operator has to be endowed with proper boundary conditions on the vertices. In general, the vertex conditions are of the form
\begin{equation*} 
    A_v \Psi(v) + B_v \Psi'(v) = 0
\end{equation*}
for all $v \in \mathcal V$, where $A_v$ and $B_v$ are $(d_v \times d_v)$-matrices and the domain of $\H$ is 
\[
    \mathcal{D}(\H) = \left\{ \psi \in H^2(\mathcal G) \mid A_v \Psi(v) + B_v \Psi'(v) = 0 \text{ for all } v \in \mathcal V \right\}.
\]

In practice, one may check the self-adjointness of $\H$
using the following classical criteria
(see e.g.\,\cite[Theorem 1.4.4.]{BeKu13}):
\begin{proposition}\label{prpHSelfAdj}
    The following assertions are equivalent:
    \begin{enumerate}
        \item The operator $\H$ is self-adjoint.
        \item For every $v \in \mathcal V$, there exist three orthogonal and mutually orthogonal
        projectors $P_{D,v}, P_{N,v}$ and $P_{R,v} = I_{d_v} - P_{D,v} - P_{N,v}$
        acting on $\C^{d_v}$ and an invertible Hermitian matrix $\Lambda_v$ acting
        on the subspace $P_{R,v} \C^{d_v}$ such that, for $\psi \in \mathcal{D}(\H)$,
        \[
            P_{D,v} \Psi(v) = 0,
            \quad P_{N,v} \Psi'(v) = 0,
            \quad P_{R,v} \Psi'(v) = \Lambda_v P_{R,v} \Psi(v).
        \]
    \end{enumerate}
\end{proposition}
In what follows, the operator $\H$ is always supposed to be self-adjoint.
As a consequence of point $2)$ in Proposition \ref{prpHSelfAdj}, this allows us to write the vertex conditions only as functions of the value of the function at the vertex, without involving the derivatives in the quadratic form. Indeed, one has
\begin{align}\label{eqQuadFormDef}
    Q_{\H}(\psi) & = \langle \H \psi, \psi \rangle_{L^2(\mathcal G)}
    = \| \psi' \|_{L^2(\G)}^2 + V(\psi)
\end{align}
where 
\begin{equation}\label{eqLocPart}
    V(\psi) \coloneqq \sum_{v \in \mathcal V} \langle \Lambda_v P_{R,v} \Psi(v), P_{R,v} \Psi(v) \rangle_{\C^{d_v}},
\end{equation}
see \cite[Equation $(1.4.38)$]{BeKu13}. Moreover, the domain of $Q_{\H}$ is given by
\begin{equation}
    \label{eqDomainQ}
    \DQH = \left\{ \psi \in H^1(\mathcal G) \mid P_{D, v} \Psi(v) = 0 \text{ for all } v \in \mathcal V \right\}.
\end{equation}
Thus, the quadratic form comprises two parts:
the $\dot H^1(\mathcal G)$-norm and a local part
coming from the vertex conditions. Due to the local one-dimensional nature of graphs
and Proposition \ref{prpHSelfAdj}, Weyl's theorem implies that the essential spectrum
of $\H$ is the same as that of $-\frac{d^2}{dx^2}$ when there is at least one infinite edge.
Moreover, the local part in the quadratic form is sequentially weakly continuous
along bounded $H^1(\mathcal{G})$ sequences.
Specifically, the following proposition holds
(see e.g.\,\cite[Lemma 2.2]{DuSh26}).
\begin{proposition} \label{prpStrongCon}
    Let $(\psi_n)_{n} \subseteq \DQH$
    be a bounded sequence
    in $H^1 (\mathcal G)$. Then, there exists $\psi \in \DQH$
    and a subsequence $(\psi_{n_k})$, such that $\psi_{n_k} \rightharpoonup \psi$
    as $n \to \infty$ weakly in $H^1(\mathcal G)$ and, for any $v \in \mathcal V$,
    $V(\psi_{n_k}) \to V(\psi)$ as $k \to \infty$.
    Moreover, if $|\mathcal{E}_{\operatorname{inf}}| > 0$, then $\sigma_{ess} ( \H ) = [0, \infty)$.
\end{proposition}

The operator \( \H \) can have at most a finite number
of negative eigenvalues, as established
in \cite[Theorem 3.7]{KoSc06},
and each of these eigenvalues belongs
to the discrete spectrum \( \sigma_{\text{dis}}(\H) \).
Recalling \eqref{eqLGamOm}, this implies that
if $|\mathcal{E}_{\operatorname{inf}}|\geq 1$, then either $l_{\H} < 0$
and $l_{\H} \in \sigma_{dis}(\H)$, or $l_{\H} = 0$.

When $l_{\H} < 0$, let $\H$ has $k \ge 1$ negative eigenvalues
$\lambda_1 \le \cdots \le \lambda_k < 0$.
We then define $E_-$ to be the vector
space spanned by the eigenfunctions associated
to those eigenvalues.
The fact that $E_-$ is finite-dimensional has several
important consequences.

\begin{lemma}
    \label{lemEquivNorms}
    There exists $\Gamma > 0$ such that,
    for all $\psi \in E_-$, one has
    $\| \psi \|_{L^{p+1}(\G)} \le \Gamma \| \psi \|_{L^2(\G)}$.
\end{lemma}

\begin{proof}
    Since $E_-$ is a finite-dimensional vector space,
    all norms on it are equivalent.
\end{proof}

\begin{lemma}[Gårding's Inequality]
    \label{lemGarding}
    There exist $C_1, C_2 > 0$ such that,
    for all $\psi \in \DQH$,
    \begin{equation*}
        \langle \H \psi, \psi \rangle_{L^2 (\G)}
        \ge C_1 \| \psi \|_{H^1 (\G)}^2 - C_2 \| \psi \|_{L^2 (\G)}^2.
    \end{equation*}
\end{lemma}

\begin{proof}
    Recalling \eqref{eqLocPart},
    there exists $\Gamma > 0$ such that one has
    \begin{equation*}
        |V(\psi)|
        \le \Gamma \sum_{v \in \mathcal V} |\Psi(v)|^2.
    \end{equation*}
    Moreover, the 1D trace inequality
    (see e.g.\,\cite[Lemma 1.3.8]{BeKu13}) implies that for any $\eps > 0$, there exists a constant $K_\eps > 0$ such that for every edge $e \in \mathcal E$ and any $\psi_e \in H^1(e)$, one has
    \begin{equation*}
        \|\psi_e\|_{L^\infty(e)}^2 \le \eps \|\psi_e'\|_{L^2(e)}^2 + K_\eps \|\psi_e\|_{L^2(e)}^2.
    \end{equation*}
    Summing this estimate over all edges $e \in \mathcal E$ (and noting that the vector norm $|\Psi(v)|^2$ is composed of the boundary values of the incident edges), we obtain
    \begin{equation} \label{eqTraceGraph}
        \sum_{v \in \mathcal V} |\Psi(v)|^2 \le 2 |\mathcal E| \left( \eps \|\psi'\|_{L^2(\G)}^2 + K_\eps \|\psi\|_{L^2(\G)}^2 \right).
    \end{equation}
    Choosing $\eps > 0$ such that
    $2 |\mathcal E| \Gamma \eps < 1$, we obtain
    \begin{equation*}
        Q_{\H}(\psi)
        \ge \|\psi'\|_{L^2 (\G)}^2 - \Gamma \sum_{v \in \mathcal V} |\Psi(v)|^2 
        \ge (1 - 2 |\mathcal E| \Gamma \eps) \|\psi'\|_{L^2 (\G)}^2
        - (2 |\mathcal E| \Gamma K_\eps) \|\psi\|_{L^2 (\G)}^2
    \end{equation*}
    which ends the proof
    setting $C_1 \coloneqq 1 - 2 |\mathcal E| \Gamma \eps > 0$
    and $C_2 \coloneqq 2 |\mathcal E| \Gamma K_\eps$.
\end{proof}

Using Gårding's Inequality,
we obtain the following criterion
for $H^1$-boundedness.
\begin{proposition} \label{prpBddEnergy}
    Let $(\psi_n)_{n} \subseteq \DQH$.
    If $(\psi_n)_{n}$ is bounded in $L^2(\G)$ and such that
    \begin{equation*}
        \limsup_{n \to \infty} E_{\H}(\psi_n) < +\infty,
    \end{equation*}
    then $(\psi_n)_n$ is bounded in $H^1(\G)$.
\end{proposition}

\begin{proof}
    Let $(\psi_n)_n \subseteq \DQH$ be a sequence
    satisfying the assumptions of the proposition.
    Applying Gårding's inequality (Lemma~\ref{lemGarding})
    to the total energy and remarking that the non-linearity is non-negative, we have
    \begin{equation*}
        C_1 \|\psi_n\|_{H^1 (\G)}^2 - C_2 \|\psi_n\|_{L^2 (\G)}^2
        \le \langle \H \psi_n, \psi_n \rangle_{L^2 (\G)}
        \le 2E_{\H}(\psi_n)
        \le 2\limsup_{n \to \infty} E_{\H}(\psi_n)
        < +\infty.
    \end{equation*}
    Since $(\|\psi_n\|_{L^2 (\G)})$ is bounded,
    so is $(\|\psi_n\|_{H^1 (\G)})$.
\end{proof}

\subsection{Properties of solutions}

We recall the following properties of solutions to \eqref{eqNls}.

\begin{proposition} \label{prpPropert}
Let $\psi\in \DQH$ to \eqref{eqNls} be any solution to  \eqref{eqNls}. Then, the following statements hold.
\begin{enumerate}
    \item We have
    \begin{equation}\label{eqNeh1}
        Q_{\H}(\psi) + \omega \| \psi \|_{L^2(\G)}^2 + \|  \psi \|_{L^{p+1}(\G)}^{p+1} = 0.
    \end{equation}
    Note that this implies that for all nonzero solutions, $\omega$ is real.
    \item If $|\mathcal{E}_{\operatorname{inf}}| > 0$, then for any $e\in \mathcal{E}_{\operatorname{inf}}$, identified with $[0,\infty)$, we have
    \begin{equation}\label{eqDecayHL}
        \lim_{x\to \infty} |\psi_e(x)| + |\psi_e'(x)| \to 0
    \end{equation}
    and, for any $x >0$, 
    \begin{equation}
        \label{eqDecayHL2}
        |\psi_e'(x)|^2 = \omega |\psi_e(x)|^2 + \frac{2|\psi_e(x)|^{p+1}}{p+1}.
    \end{equation}
    \item For any $q \in [2,\infty]$ and any $e \in \mathcal{E}$, we have $\psi_e \in W^{3,q}(\mathring{e})$.
    \item For any $e \in \mathcal{E}_{\operatorname{inf}}$,
    either $\psi_e \equiv 0$ or $\psi_e(x) \neq 0$
    for all $x \in e$.
\end{enumerate}
\end{proposition}
The proof of Proposition \ref{prpPropert} is standard,
see for example \cite[Lemma $3.2$]{FuOhOz08}.
The last point follows from the classification
of all solutions to the stationary equation on the half-line,
which we recall below.

\begin{proposition}\label{prpHLSol}
    The only $H^1(\R^+)$ solutions (up to a complex phase shift) to \eqref{eqNls} are given by
    \begin{itemize}
        \item for $\omega > 0$,
        \begin{equation}\label{eqSol1}
            \phi_{\omega}(x) = \left(\frac{(p+1)\omega}{2} \right)^\frac{1}{p-1} \left(\sinh \left(\frac{(p-1) \sqrt{\omega}}{2}x + a \right) \right)^{-\frac{2}{p-1}}
        \end{equation}
        with $a = a (\omega,p, \H) > 0$;
        \item for $\omega = 0$ and $p \in (1,5)$,
        \begin{equation}\label{eqSol2}
            \phi_0 (x) = \left( \frac{2(p + 1)}{(p - 1)^2} \right)^{\frac{1}{p - 1}} \left( x + b \right)^{-\frac{2}{p - 1}} 
        \end{equation}
        with $b = b(p,\H) >0$.
    \end{itemize}
    For $\omega < 0$ and $p>1$ or $\omega = 0$ and $p \geq 5$, the only solution is $\phi \equiv 0$. 
\end{proposition}

\begin{proof}
For the cases $\omega >0$, and $\omega = 0$, one can verify directly that \eqref{eqSol1} and \eqref{eqSol2} are solutions to \eqref{eqNls} and their uniqueness
follows from \cite[Theorem $1$ and Theorem $2$]{KaOh09}.

Now, suppose $\omega < 0$ and assume that \eqref{eqNls} admits a solution $\psi$. Then, from \eqref{eqDecayHL} and \eqref{eqDecayHL2}, there exists $R >0$ big enough such that for any $x \in e$ such that $x >R$, we have 
\begin{equation*}
     |\psi_e'(x)|^2 =  |\psi_e(x)|^2 \left(\omega + \frac{2|\psi_e(x)|^{p-1}}{p+1}\right) < 0,
\end{equation*}
which leads to a contradiction.
\end{proof}

We also have an upper bound on $\omega$ in \eqref{eqNls} for the existence of nontrivial solutions.

\begin{proposition}\label{prpOmegaUpBnd}
    If $\omega \geq -l_{\H}$ where $l_{\H}$
    is defined in \eqref{eqLGamOm},
    then the only solution to \eqref{eqNls} is $\psi \equiv 0$. 
\end{proposition}

\begin{proof}
    Suppose that $\psi$ is a solution to \eqref{eqNls}.
    If $\omega \geq -l_{\H}$ then by \eqref{eqNeh1}, we have
    \begin{equation}
        0 \geq (\omega + l_{\H}) \| \psi \|_{L^2(\G)}^2
        + \| \psi \|_{L^{p+1}(\G)}^{p+1}
        \geq \| \psi \|_{L^{p+1}(\G)}^{p+1},
    \end{equation}
    implying $\psi \equiv 0$. 
\end{proof}

\section{Energy ground states}
\label{secEGS}

In this section, we analyze the minimization problem \eqref{eqEnMin} through a sequence of steps. We begin by showing that the energy admits a lower bound for any mass $\mu$. We then investigate the properties of the function $\mu \mapsto \tau_{\mu}$, proving in particular that minimizers do not exist for all values of $\mu$ in the noncompact case, with the outcome depending on the exponent $p$, $\mathcal{G}$, and $\H$. Since we work here with a general operator \(\H\), only partial results can be obtained, as the behavior may strongly depend on the vertex conditions. However, further in this section, we demonstrate that if $\H$ has only $\delta$-vertex conditions, then we can give a full picture of existence, both when $p \geq 5$ and when $p < 5$, also assuming that $|\mathcal{V}| = 1$ in this case. Finally, we show that in the small mass regime, energy ground states can be obtained as bifurcations from linear ground states.

\subsection{General case}

We start by establishing a lower bound on $\tau_\mu$. To this end, we introduce the global (unconstrained) minimization problem
\begin{equation} \label{eqEnMinGlob}
    \tau_{\min} = \inf \left\{ E_{\H}(\psi) \mid \psi \in \DQH \right\}.
\end{equation}
As will become clear, this problem does not always admit a minimizer. In fact, we will show that the behavior of the map $\mu \mapsto \tau_\mu$ defined in \eqref{eqEnMin} is heavily related to the existence of such a global minimizer. We begin by proving that the problem is indeed bounded from below.

\begin{lemma}\label{lemMinGlob}
     If $l_{\H} \geq 0$, then $\tau_{\operatorname{min}} = 0$, while if $l_{\H} < 0$, then $0 > \tau_{\operatorname{min}} > -\infty$.
     In particular, for any $\mu > 0$, $\tau_{\mu} \geq \tau_{\operatorname{min}} >-\infty$. 
\end{lemma}

\begin{proof}
    Let $\psi \in \DQH$. Let $v \in \mathcal{V}$ and $e \in J_{v}$, identified with $[0, A_e]$ for some $A_e >0$ or $[0, \infty)$.
    By the definition of $\Llow$ in \eqref{eqMinLnght}, it follows that $[0, \Llow] \subset e$.
    Then for any $x \in [0, \Llow]$ we have
    \begin{equation*}
        | \psi_e (0) |^2 = |\psi_e(x)|^2 - \int_0^{x} \partial_y |\psi_e(y)|^2dy \leq |\psi_e(x)|^2 + 2 \int_0^{x} |\psi_e'(y) \psi_e(y)|dy.
    \end{equation*}
    It follows after integration in $x$ and Young inequality that for any $\eps >0$, there exists $C_\eps >0$ depending on $\eps$ such that
    \begin{equation}\label{eqLocalGN}
        | \psi_e (0) |^2 \lesssim \|\psi_e\|^2_{L^2{(0,\Llow)}}  + \|\psi_e\|_{L^2{(0,\Llow)}}\|\psi_e'\|_{L^2{(0,\Llow)}}
        \leq C_{\eps} \| \psi_e \|_{L^2(0,\Llow)}^2 + \eps \| \psi_e' \|_{L^2(0,\Llow)}^2.
    \end{equation}
    Repeating this computation on every vertex and edge, we get
    \begin{align*}
        \sum_{v \in \mathcal{V}}  \langle \Lambda_v P_{R,v} \Psi(v), P_{R,v} \Psi(v) \rangle_{\R^{d_v}}
        &\lesssim \sum_{v\in\mathcal{V}} \sum_{e,e' \in J_v}|\psi_e(v)||\psi_{e'}(v)| \\
        &\lesssim \sum_{v \in \mathcal{V}} \sum_{e \in J_v}| \psi_e(v) |^2
        \leq \sum_{v\in \mathcal{V}} \sum_{e \in J_v} C_\eps \| \psi_e \|_{L^2(v,v+L)}^2 + \eps \| \psi_e' \|_{L^2(v,v+L)}^2,
    \end{align*}
    were we slightly abuse notations identifying each edge $e\in J_{v}$ with
    an interval $[v, v+A_e]$ for some $A_e >0$ and $[v, \infty)$. 
    By Hölder's and Young's inequalities, there exists $R = R(L,\eps) > 0$ such that
    \[
        C_\eps \| \psi_e \|_{L^2(v,v+L)}^2 \leq \frac{1}{p+1} \| \psi_e \|_{L^{p+1}(v,v+L)}^{p+1} + R.
    \]
    Since $|\mathcal{V}| < \infty$ and $|\mathcal{E}| <\infty$ and recalling that the quadratic form is defined by \eqref{eqQuadFormDef}, there exists $K = K(|\mathcal{V}|, |\mathcal{E}|) > 0$ such that 
    \begin{equation*}
        E_{\H}(\psi) \geq \left( \frac{1}{2} - \eps \right) \| \psi' \|_{L^2(\G)}^2 - KR \geq  - KR
    \end{equation*}
    which yields the result, taking an arbitrary $\eps \in (0, 1/2).$ 

    It remains to prove that $\tau_{\operatorname{min}} < 0$ when $l_{\H} <0$. Let $\phi$ be such that $\| \phi \|_{L^2(\G)} = 1$ and $Q_{\H}(\phi) = l_{\H}$.
    Then for any $t >0$, we have 
    \begin{equation*}
        E_{\H} (t \phi) = \frac{t^2}{2} l_{\H} + \frac{t^{p+1}}{p+1} \|\phi \|_{L^{p+1}(\G)}^{p+1}.
    \end{equation*}
    In particular, for $t$ small enough, we have $ E_{\H} (t \phi) <0$. 
\end{proof}

\begin{remark}
    On compact graphs, the previous lower bound implies that energy ground states exist for all masses.
    Moreover, if $\G$ is not compact and $l_{\H} = 0$, then one has $\tau_\mu \ge 0$ for all $\mu > 0$,
    and considering functions of the form $u_n = n^\frac{1}{2} u(nx)$,
    compactly supported on a fixed half-line,
    shows that in fact $\tau_\mu = 0$ for all $\mu > 0$,
    so that energy ground states do not exist.
\end{remark}

In what follows, we focus on the case $l_{\H} < 0$ and $|\mathcal{E}_{\operatorname{inf}}| > 0$. In this setting, the existence of an energy minimizer is not guaranteed due to the lack of compactness of the embedding $H^1(\G) \hookrightarrow L^2(\G)$. One must exploit the concentration-compactness theorem adapted for graphs, see \cite[Lemma 3.7]{CaFiNo17}. Out of the three noncompactness scenarios, vanishing and escaping will be quickly ruled out, as we will show that the minimal levels of the energy are always negative for any positive mass. But, a minimizing sequence for $\tau_{\mu}$ may undergo the dichotomy scenario. In particular, some mass can be expelled to infinity in the weak limit.

To determine whether such behavior occurs, we begin by analyzing general properties of the function $\mu \mapsto \tau_\mu$. Indeed, if this function were strictly decreasing, the dichotomy scenario would be energetically unfavorable, since any loss of mass would lead to an increase in energy. However, such a strong monotonicity cannot be guaranteed in general, and we can only assert that the function is non-increasing.

\begin{lemma}\label{lemContMon}
    Let $|\mathcal{E}_{\operatorname{inf}}| >0$ and $l_{\H} <0$.
    Then the function $\mu \mapsto \tau_{\mu}$ is continuous, non-increasing, and $\tau_{\mu} < 0$ for any $\mu > 0$. 
\end{lemma}

\begin{proof}
    We start by showing the continuity. Let $\mu >0$ and $(\mu_n)_n \subseteq \R^+$ such that $\mu_n \to \mu$ as $n\to \infty$. Let $(u_n)_n \subseteq  H^1(\G)$ be such that $\| u_n \|_{L^2(\G)} = \mu_n$ and 
    \begin{equation*}
        E_{\H}(u_n) \leq \tau_{\mu_n} + \eps
    \end{equation*}
    for some $\eps >0$. Let $v_n = \frac{\mu}{\mu_n} u_n$. 
    Then 
    \begin{equation*}
        E_{\H}(v_n) = \frac{\mu^2}{\mu_n^2} E_{\H}(u_n) + \left( \frac{\mu^{p+1}}{\mu_n^{p+1}}-\frac{\mu^2}{\mu_n^2}   \right) \frac{\| u_n \|_{L^{p+1}(\G)}^{p+1}}{p+1}  
    \end{equation*}
    Using Proposition~\ref{prpBddEnergy}, we deduce that $(u_n)_n$ is uniformly bounded in $H^1(\G)$.
    Hence, there exists $C >0$ such that
    \begin{equation*}
          \tau_{\mu} \leq E_{\H}(v_n) \leq \frac{\mu^2}{\mu_n^2} (\tau_{\mu_n} + \eps) +  \left( \frac{\mu^{p+1}}{\mu_n^{p+1}}-\frac{\mu^2}{\mu_n^2}   \right) C.
    \end{equation*}
    Since $\eps$ is arbitrary, taking the limit $n\to \infty$ above implies 
    \begin{equation*}
         \tau_{\mu} \leq \liminf_{n\to \infty} \tau_{\mu_n}.
    \end{equation*}
    On the other hand, let $u$ be such that $\| u \|_{L^2(\G)} = \mu$.
    Let $w_n \coloneqq \frac{\mu_n}{\mu} u$. Then 
    \begin{equation*}
        \tau_{\mu_n} \leq  E_{\H}(w_n) = \frac{\mu_n^2}{\mu^2} E_{\H}(u) + \left( \frac{\mu_n^{p+1}}{\mu^{p+1}}-\frac{\mu_n^2}{\mu^2}   \right) \frac{\| u \|_{L^{p+1}(\G)}^{p+1}}{p+1}  
    \end{equation*}
    which implies that
    \begin{equation*}
        \limsup_{n\to \infty} \tau_{\mu_n} \leq  E_{\H}(u).
    \end{equation*}
    Taking the infimum over all $u \in H^1(\G)$ with $\| u \|_{L^2(\G)} = \mu$ we obtain the continuity.

    We observe that there exists $\mu_* >0$
    such that $\tau_{\mu} <0$ for any $\mu \in (0, \mu_*)$.
    Indeed, if $\phi$ is such that $\| \phi \|_{L^2(\G)} = 1$
    and $Q_{\H}(\phi) = l_{\H}$, then for $\mu$ small enough,
    $\tau_\mu \le E_{\H}(\mu\phi) <0 $.

    Now we show the monotonicity. Let $e \in \mathcal{E}_{\operatorname{inf}}$ identified with $[0,\infty)$.
    Let $u \in H^1(\G)$
    be such that $\| u \|_{L^2{(\G)}} = \mu$ and $u_e$ is compactly supported.
    Let $v_n = (v_{n, e})_{e \in \mathcal{E}}$ such that $v_n \equiv 0$ on $\mathcal{E} \setminus \{ e \}$
    and $v_{n,e} \in H^1(e) \cap C_c (e)$ is given by
    \begin{equation*}
        v_{n, e} (x) = 
        \begin{cases}
            0   & \text{ if } x \le n^2-1 \text{ or } 3n^2 + 1 < x,\\
            c_n(x - n^2 + 1)    & \text{ if } n^2-1 < x \le n^2,\\
            c_n & \text{ if } n^2 < x \le 3n^2, \\
            c_n(-x + 3n^2 + 1) & \text{ if } 3n^2 < x \le 3n^2 + 1,\\
        \end{cases}
    \end{equation*}
    for $c_n$ chosen so that $\| v_n \|_{L^2(\G)} = \nu$ for some $\nu > 0$. It is clear that $c_n$  is proportional to $\nu n^{-1}$. Let $y_n > 0$ be such that $\supp(u_e) \cap \supp (v_{n, e} (\cdot -y_n)) = \emptyset$ for any $n$. Finally, let $w_n = (w_{n, e})_{e \in \mathcal{E}}$ such that $w_n \coloneqq u$ on $\mathcal{E} \setminus \{ e \}$ and $w_{n, e} \coloneqq u_e + v_{n, e} (\cdot - y_n)$. Then
    \begin{equation*}
        \tau_{\mu + \nu} \leq E_{\H} (w_n) = E_{\H}(u) + E_{\H}(v_n).
    \end{equation*}
    Since $E_{\H}(v_n) \to 0$ as $n \to \infty$,
    we obtain $\tau_{\mu + \nu} \leq E_{\H}(u)$ by taking limits.
    Taking the infimum over all functions
    $u \in H^1(\G)$ compactly supported on $e$ such that
    $\| u\|_{L^2{(\G)}} = \mu$ and using the dense inclusion $H^1(\G) \cap C_c (e) \hookrightarrow H^1 (e)$, we get
    \begin{equation*}
        \tau_{\mu + \nu} \leq \tau_{\mu}.
    \end{equation*}
    Since $\tau_{\mu} <0$ for $\mu \in (0,\mu_*)$ and $\mu \mapsto \tau_\mu$ is non-increasing, we conclude that $\tau_{\mu} <0$ for all $\mu >0$.
\end{proof}

\begin{remark}\label{rmkMonComp}
    If $|\mathcal{E}_{\operatorname{inf}}| =0$ and $l_{\H} <0$, only the continuity part of Lemma \ref{lemContMon} is true. Indeed, let $\psi = (\psi_e)_{e \in \mathcal{E}} \in H^1(\G)$ be any function. Let $e$ be any edge, which we identify with $[0, M_e]$ for some $M_e >0$. Then by \eqref{eqLocalGN}, for any $\eps>0$, there exists $C_e = C_e (\eps) > 0$ such that
    \begin{equation*}
        |\psi_e(0)|^2 +  |\psi_e(M_e)|^2 \leq C_e \| \psi_e \|_{L^2([0,M_e])}^2 + \eps \| \psi_e \|_{L^2([0,M_e])}^2.
    \end{equation*}
    On the other hand, by H\"older's and Young's inequalities, there exists $K_e > 0$ such that 
    \begin{equation*}
        \| \psi_e \|_{L^2([0,M_e])}^{p+1} \leq K_e \| \psi_e \|_{L^{p+1}([0,M_e])}^{p+1}.
    \end{equation*}
    Thus, by summing on all edges, we obtain that there exists $C = C (\eps) > 0$ and $K >0$ such that
    \begin{equation*}
          E_{\H}(\psi) \geq \left(\frac{1}{2} - \eps\right) \| \psi' \|_{L^2(\G)}^2 - C \| \psi \|_{L^2(\G)}^2 + K\| \psi \|_{L^{2}(\G)}^{p+1}.
    \end{equation*}
    In particular, for $\eps = \frac{1}{4}$ and $\mu = \| \psi \|_{L^2}$ we obtain 
    \begin{equation} \label{eqMonComp}
          E_{\H}(\psi) \geq - C \mu^2 + \frac{1}{p+1} \mu^{p+1}.
    \end{equation}
    that is, $\tau_\mu \to +\infty$ as $\mu \to \infty$. 
\end{remark}

We now show that a minimizer exists for all values of $\mu$ within an interval where the function $\mu \mapsto \tau_\mu$ is strictly decreasing. On the other hand, no conclusion can yet be drawn regarding the existence of minimizers in intervals where $\tau_\mu$ is constant.

\begin{lemma}\label{lemExistMin}
     Let $|\mathcal{E}_{\operatorname{inf}}| >0$ and $l_{\H} <0$.
     If $\tilde \mu > 0$ satisfies
     $\tau_{\tilde \mu} < \tau_\mu$ for all $\mu \in [0,\tilde \mu)$, then $\tau_{\tilde \mu}$ admits a minimizer.
\end{lemma}

\begin{proof}
    Let $(\psi_n)_n$ be a minimizing sequence
    for $\tau_{\tilde{\mu}}$.
    By Proposition \ref{prpBddEnergy},
    it is uniformly bounded in $H^1(\G)$.
    Hence, by Proposition \ref{prpStrongCon},
    there exists a subsequence,
    still denoted by $(\psi_n)_n$, and $\psi \in \DQH$ such that $\psi_n \rightharpoonup \psi$
    in $H^1(\G)$ and $V(\psi_n) \to V(\psi)$, where $V$ is defined in \eqref{eqLocPart}.
    Thus, by the weak lower-semicontinuity for the $\dot H^1(\G)$ and the $L^{p+1}(\G)$ norms, we have 
    \begin{equation*}
        E_{\H}(\psi) \leq \liminf_{n\to\infty} E_{\H}(\psi_n) = \tau_{\tilde \mu},
    \end{equation*}
    and
    \begin{equation*}
        0 \leq \| \psi \|_{L^2 (\G)} = \nu \leq \tilde \mu.
    \end{equation*}
    But since $\tau_{\nu} \leq E_{\H} (\psi)$
    and by hypothesis $\tau_{\tilde \mu} < \tau_{\nu}$ if $\nu < \tilde \mu$,
    we obtain that $\nu = \tilde \mu$, so that $\psi$ is a minimizer.
\end{proof}
Now, we show that the intervals where the map $\mu \mapsto \tau_{\mu}$
is strictly decreasing correspond to those for which the associated minimizer satisfies equation \eqref{eqNls} with a positive Lagrange multiplier $\omega$. 

\begin{lemma} \label{lemPosMult}
    Let $|\mathcal{E}_{\operatorname{inf}}| >0$ and $l_{\H} <0$.
    Let $\psi \in \DQH$ be a minimizer for $\tau_{\tilde \mu}$ and a  solution to \eqref{eqNls} for some $\omega > 0$. Then, for all $\mu > \tilde \mu$, we have $\tau_\mu < \tau_{\tilde \mu}$.
\end{lemma}

\begin{proof}
    For any $t > 0$, observe that $\| (1 + t) \psi \|_{L^2(\G)} = (1 + t) \tilde{\mu}$ while 
    \begin{equation*}
         E_{\H} ((1 + t) \psi)
         = \frac{(1 + t)^2}{2} Q_{\H} (\psi) + \frac{(1 + t)^{p+1}}{p+1} \| \psi \|_{L^{p+1} (\G)}^{p+1}.
        \end{equation*}
        Consequently, by \eqref{eqNls}, it follows that
        \begin{align*}
            \frac{d}{dt} E_{\H} ((1 + t) \psi )
            & = (1 + t) Q_{\H} (\psi) + (1 + t)^{p}  \| \psi \|_{L^{p+1} (\G)}^{p+1} \\
            & = -\omega \| \psi \|_{L^{2} (\G)}^{2} + t Q_{\H} (\psi) + \left( (1 + t)^{p} - 1 \right) \| \psi \|_{L^{p+1}(\G)}^{p+1}.
        \end{align*}
        Thus, as $t \to 0$, we obtain 
        \begin{equation*}
            \frac{d}{dt} E_{\H} ((1 + t) \psi)
            = - \omega \tilde \mu + o (t),
        \end{equation*}
        which implies that there exists $t^* >0$ such that for $t \in (0, t^*)$, 
        \begin{equation*}
            \tau_{(1 + t) \tilde \mu}
            \leq  E_{\H} ((1 + t) \psi)
            <  E_{\H} (\psi)
            = \tau_{\tilde \mu}.
        \end{equation*}
        The result follows from Lemma \ref{lemContMon}.
\end{proof}

\begin{corollary}\label{rmkOmegaZero}
    Let $|\mathcal{E}_{\operatorname{inf}}| >0$ and $l_{\H} <0$.
    Let $u$ be a minimizer for $\tau_{\tilde \mu}$ with $ \tilde \mu >0$. If there exists $\mu > \tilde \mu$ satisfying $\tau_{\tilde \mu} = \tau_\mu$, then $u$ is a solution to \eqref{eqNls} for some $\omega \leq 0$.
\end{corollary}

In particular, the intervals $I$ of mass values $\mu$ for which the minimization problem defining $\tau_{\mu}$ may fail to admit a minimizer are necessarily of the form $I = (\mu_I, \mu_F)$, $I = (\mu_I, \mu_F]$ or $I = (\mu_I, \infty)$ for some initial mass $\mu_I > 0$. At this threshold, a minimizer $u_I$ exists for $\tau_{\mu_I}$ and satisfies the equation \eqref{eqNls} with a non-positive Lagrange multiplier, i.e., $\omega \leq 0$. Moreover, the function $\mu \mapsto \tau_{\mu}$ is constant on $I$. This is a crucial structural feature of the problem: as long as the minimizer corresponds to a strictly positive $\omega$, the strictly decreasing nature of $\mu \mapsto \tau_\mu$ ensures compactness and thus existence. In contrast, once $\omega$ becomes non-positive, the function $\mu \mapsto \tau_\mu$ may become constant, which opens the possibility of dichotomy and loss of compactness, hence the non-attainment of the infimum.

In the next proposition, we show that there exists an initial interval $[0,\mu_{1})$ such that, for any $\mu \in [0,\mu_{1})$, the minimum defining $\tau_\mu$ is attained. 
In particular, the function $\mu \mapsto \tau_\mu$ is strictly decreasing on this interval.

\begin{proposition}\label{prpMuOne}
Let $|\mathcal{E}_{\operatorname{inf}}| >0$ and $l_{\H} <0$. Then there exists $\mu_1 >0$ such that a minimizer for $\tau_\mu$ exists for any $\mu \in (0,\mu_1]$.    
\end{proposition}

\begin{proof}
    Suppose by contradiction that there exists a sequence $\mu_n \to 0$ as $n\to \infty$ such that a minimizer for $\tau_{\mu_n}$ does not exist. Then, by Lemmas \ref{lemExistMin} and \ref{lemContMon}, there exists $\nu_n \to 0$, $\eps_n >0$ such that $\tau_\mu$ is constant for any $\mu \in [\nu_n, \mu_n]$ and
    satisfies $\tau_{\nu_n} < \tau_{\mu}$
    for all $\mu \in (0, \nu_n)$. Then, by Lemma \ref{lemExistMin}, $\tau_{\nu_n}$ admit minimizers which we denote by $\psi_n$, and by Corollary \ref{rmkOmegaZero} these minimizers satisfy \eqref{eqNls} with $\omega \leq 0$. This implies that
    \begin{equation*}
       \omega \| \psi_n \|_{L^2 (\G)}^2 + Q_{\H} (\psi_n)
       = - \| \psi_n \|_{L^{p+1} (\G)}^{p+1} 
    \end{equation*}
    which yields by the definition of $l_{\H}$ in \eqref{eqLGamOm}
    \begin{equation}\label{eqIneq1}
         \tau_{\nu_n}
         = E_{\H}(\psi_n)
         = \left( \frac{1}{2} - \frac{1}{p+1}\right) Q_{\H} (\psi_n) - \frac{\omega}{p+1}\| \psi_n \|_{L^2 (\G)}^2
         \geq \left( \frac{1}{2} - \frac{1}{p+1}\right) l_{\H} \nu_n^2 - \frac{\omega}{p+1} \nu_n^2.
    \end{equation}
    On the other hand, let $\phi$ be such that $\| \phi \|_{L^2(\G)} = 1$ and $Q_{\H}(\phi) = l_{\H}$. Then we have 
    \begin{equation*}
        \tau_{\nu_n}  \leq E_{\H}(\nu_n \phi) =  \frac{1}{2} l_{\H} \nu_n^2 + \frac{\nu_n^{p+1}}{p+1}  \| \phi \|_{L^{p+1}(\G)}^{p+1}
    \end{equation*}
    Combined with \eqref{eqIneq1}, this yields 
    \begin{equation*}
        0 < - l_{\H} - \omega \leq \nu_n^{p-1}  \| \phi \|_{L^{p+1}(\G)}^{p+1}
    \end{equation*}
    which yields a contradiction for $\nu_n$ small enough.
\end{proof}
Having established the initial properties of the curve $\mu \mapsto \tau_\mu$, we now study its behavior for large masses. A distinction must be made between the cases $p \in (1,5)$ and $p \geq 5$, as the behavior of the function differs in these two regimes.

In particular, in the next part, we show that, in the case $p \in (1,5) $, there exists a value $\mu_{2}>0$ such that $\tau_{\mu}$ does not admit a minimizer for any $\mu > \mu_{2}$. This implies that the function $\mu \mapsto \tau_\mu$ admits a terminal interval $(\mu_2, \infty)$ on which it is constant and no minimizers exist.

This result is related to the existence of a global minimizer of the energy over $\DQH$ (in particular, a solution of \eqref{eqNls} with $\omega = 0$). We first prove that a candidate $\tilde{\psi} \in \dot{H}^1(\G) \cap L^{p+1}(\G)$ exists independently of the specific structure of $\G$ and the operator $\H$, as stated in the following proposition.

\begin{proposition}\label{prpGlbMin}
    Let $|\mathcal{E}_{\operatorname{inf}}| > 0$ and $l_{\H} < 0$.
    Then there exists $\tilde{\psi} \in \dot H^1(\G) \cap L^{p+1}(\G)$ such that $E_{\H} (\tilde{\psi}) = \tau_{\operatorname{min}}$.
\end{proposition}

In order to prove this proposition, we first need to adapt Proposition \ref{prpStrongCon} in the following way.

\begin{lemma}\label{lemStrConv2}
    For any $p > 1$, let
    $(\tilde{\psi}_n)_n \subseteq
    W_p (\G) \coloneqq \dot H^1(\G) \cap L^{p+1}(\G)$
    be bounded with respect to the norm
    \begin{equation*}
        \| \tilde{\psi} \|_{W_p (\G)}
        \coloneqq \| \tilde{\psi}' \|_{L^2(\G)}
        + \| \tilde{\psi} \|_{L^{p+1}(\G)}. 
    \end{equation*}
    Then there exists $\tilde{\psi} \in W_p$
    and a subsequence $(\tilde{\psi}_{n_k})$,
    such that $\tilde{\psi}_{n_k} \rightharpoonup \tilde{\psi}$ as $n \to \infty$ weakly in $ H^1(\G) \cap L^{p+1}(\G)$ and
    $V(\tilde{\psi}_{n_k}) \to V(\tilde{\psi})$ as $k \to \infty$.
\end{lemma}

\begin{proof}
    If $\G$ is compact, this follows directly from Proposition \ref{prpStrongCon} and H\"older's inequality. Let us then suppose that $\G$ is not compact.
    We show that $W_p \hookrightarrow L^{\infty}(\G)$. Let $e \in \mathcal{E}$ be any edge and $x\in e$. Suppose $e$ is an infinite edge (as the case $e$ finite is similar and easier). We identify $e$ as $(-\infty,0]$.
    Then we have 
    \begin{equation*}
        |\tilde{\psi}_e(x)|^{\frac{p+3}{2}}
        = \int_{-\infty}^x \frac{p+3}{2} |\tilde{\psi}_e (y)|^{\frac{p+1}{2}} \tilde{\psi}_e' (y) \operatorname{sign} (\tilde{\psi} (y)) dy.
    \end{equation*}
    By Cauchy-Schwarz, we then get
    \begin{equation*}
        |\tilde{\psi}_e (x)|
        \lesssim \left( \| \tilde{\psi}_e\|_{L^{p+1} (-\infty, 0)}^{\frac{p+1}{2}} \| \tilde{\psi}_e' \|_{L^{2} (-\infty, 0)} \right)^{\frac{2}{p+3}}.
    \end{equation*}
    Consequently, we get by summing on all edges
    \begin{equation}
        \| \tilde{\psi} \|_{L^{\infty}(\G)}
        \lesssim \| \tilde{\psi} \|_{W_p (\G)}
    \end{equation}
    Next, we show that $H^1(\G) \cap C_c(\G)$ is dense in $W_p$. The proof is classical, so we just give the idea. Let \( u \in \dot{H}^1(\G) \cap L^{p+1}(\G) \). For \( R > 0 \), let \( \chi_R \in C_c^\infty(\mathbb{R^+}) \) be a smooth cut-off function such that
\[
\chi_R(x) =
\begin{cases}
1 & \text{if } |x| \leq R, \\
0 & \text{if } |x| \geq R+1,
\end{cases}
\quad \text{and} \quad |\chi_R'| \leq C.
\]
Define \( u_R\coloneqq H^1(\G) \cap L^{p+1}(\G)\) where we applied the cut-off on infinite edges. Then:
\[
\| u_R' - u' \|_{L^2(\G)} \leq \| (1 - \chi_R) u' \|_{L^2(\G)} + \| \chi_R' u \|_{L^2(\G)}.
\]
The first term vanishes as \( R \to \infty \) by dominated convergence, since \( \partial_x u \in L^2(\G) \). For the second term, we observe that, if $e\in \mathcal{E}_{\operatorname{inf}}$, then by H\"older's inequality we get 
\[
\| \chi_R' u_e \|_{L^2}^2 \leq C^2 \int_{R \leq |x| \leq R+1} |u_e(x)|^2 dx\lesssim \left(\int_{R \leq |x| \leq R+1} |u_e(x)|^{p+1} dx\right)^{\frac{2}{p+1}} \to 0 \quad \text{as } R \to \infty.
\]
This clearly implies that $\| \chi_R' u \|_{L^2(\G)} \to 0$ as $R \to \infty$. 
Similarly,
\[
\| u_R - u \|_{L^{p+1}} = \| (1 - \chi_R) u \|_{L^{p+1}} \to 0 \quad \text{as } R \to \infty,
\]
by the dominated convergence theorem. Finally, by a standard mollification technique, we can obtain a smooth sequence \( u_{R,\varepsilon} \coloneqq \rho_\varepsilon * u_R \in C_c^\infty(\mathbb{R}) \) with the same characteristics. 
\end{proof}

\begin{proof}[Proof of Proposition \ref{prpGlbMin}]
    If $l_{\H} \geq 0$, then clearly $\tilde{\psi} \equiv 0$.
    Suppose $l_{\H} < 0$ and let $(\tilde{\psi}_n)_n$
    be a minimizing sequence for $\tau_{\operatorname{min}} <0$.
    As it is bounded in $\dot{H}^1(\G) \cap L^{p+1}(\G)$,
    Lemma \ref{lemStrConv2} implies that
    there exists a subsequence, still denoted by $(\tilde{\psi}_n)_n$,
    and $\tilde{\psi} \in \dot{H}^1(\G) \cap L^{p+1}(\G)$
    such that $\tilde{\psi}_n \rightharpoonup \tilde{\psi}$
    weakly in $\dot{H}^1(\G) \cap L^{p+1}(\G)$.
    By weak lower semicontinuity of the norms,
    we deduce that $\tilde{\psi}$ satisfies $E_{\H} (\tilde{\psi}) = \tau_{\operatorname{min}}$.
\end{proof}

\begin{remark} \label{rmkMConcentr}
    We highlight the fact that $\tilde{\psi}$ found in Proposition \ref{prpGlbMin} may not, in general, belong to $L^2(\G)$. The failure to lie in this space requires mass accumulation outside of the compact core $\mathcal{K}$ defined in \eqref{eqCompCore}. Indeed, suppose that a minimizing sequence $(\tilde{\psi}_n)_{n}$ satisfies $\mu_n = \| \tilde{\psi}_n \|_{L^2(\G)} \to \infty$ as $n \to \infty$. Assume by contradiction that $\| \tilde{\psi}_n \|_{L^2(\mathcal K)} \to \infty$. Then, as in Remark \ref{rmkMonComp}, we have that
    \begin{align*}
        E_{\H} (\tilde{\psi}_n)
        & \geq \frac{1}{2}\| \tilde{\psi}_n' \|_{L^{2} (\G \setminus \mathcal{K})}^2  + \frac{1}{p+1} \| \tilde{\psi}_n \|_{L^{p+1} (\G \setminus \mathcal{K})}^{p+1} \\
        & \quad + \frac{1}{4} \| \tilde{\psi}_n' \|_{L^2 (\mathcal{K})}^2 - C \| \tilde{\psi}_n \|_{L^2(\mathcal{K})}^2 + K \| \tilde{\psi}_n \|_{L^{2} (\mathcal{K})}^{p+1}.
    \end{align*}
    for some $C, K >0$. Thus, as $n \to \infty$, it follows that $E_{\H} (\tilde{\psi}_n) \to \infty$ as $n \to \infty$ which leads to a contradiction. Hence, we must have $\| \tilde{\psi}_n \|_{L^2(\G \setminus \mathcal{K})} \to \infty$ as $n \to \infty$.
\end{remark}
We obtain the following corollary, for which we recall that the minimization problem $\tau_{\operatorname{min}}$ is defined in \eqref{eqEnMinGlob}.
\begin{corollary}\label{corFnlPlt1}
    Let $|\mathcal{E}_{\operatorname{inf}}| >0$ and $l_{\H} <0$. Then, for $p \in (1,5)$, $\tau_{\operatorname{min}}$ admits a minimizer. 
\end{corollary}

\begin{proof}
    Let $\tilde{\psi} \in \dot{H}^1 (\G) \cap L^{p+1} (\G)$ be the minimizer found in Proposition \ref{prpGlbMin}. In particular, as it is a global minimizer, it satisfies \eqref{eqNls} with $\omega = 0$. By uniqueness of the decaying solutions, this implies that either $\tilde{\psi}_e \equiv 0$ or it is of the form \eqref{eqSol2} on any $e\in \mathcal{E}_{\operatorname{inf}}$. In the first case, $\supp (\tilde{\psi})$ is compact, which implies that $\tilde{\psi} \in H^1(\G)$. The same is true in the second case, as the function in \eqref{eqSol2} belongs to $L^2(\R^+)$ since $1 < p < 5$.

    In particular, the global minimum of the energy is reached in a function of $\DQH$. 
\end{proof}

Next, we consider the case $p \geq 5$.
A direct consequence of Propositions \ref{prpHLSol}
and \ref{prpGlbMin} is the following.

\begin{proposition} \label{prpOmegaZero}
    Let $|\mathcal{E}_{\operatorname{inf}}| >0$ and $l_{\H} <0$.
    Let $p \geq 5$. If a minimizer $u = (u_e)_{e \in \mathcal{E}} \in \DQH$ for problem \eqref{eqEnMinGlob} exists, then for all $e \in \mathcal E_{\operatorname{inf}}$, $u_e \equiv 0$.
\end{proposition}

\begin{proof}
    It follows from the fact that functions in \eqref{eqSol2} are not in $L^2(\R^+)$ for $p \geq 5$.
\end{proof}

Proposition \ref{prpOmegaZero} is not sufficiently strong to conclude that a minimizer does not exist. The problem lies in the absence of continuity at the vertices for general coupling conditions. We will show in the next section that when continuity is assured, no global minimizer exists.

Next, we study the behavior for large masses. We define $\mu_{\operatorname{min}} \in (0, \infty]$ as
\begin{equation} \label{eqMinMass}
    \mu_{\operatorname{min}} \coloneqq 
    \begin{cases}
        \min \{ \mu > 0: \tau_\mu = \tau_{\operatorname{min}} \} & \text{ if } \mu \mapsto \tau_{\mu} \text{ admits a minimum}, \\
        \infty & \text{ if not.}
    \end{cases}
\end{equation}

By Lemma \ref{lemExistMin}, if $\mu_{\operatorname{min}} \neq \infty$, then $\tau_{\operatorname{min}}$ admits a minimizer.
We deduce the following result, ensuring that, in this case, there cannot exist a sequence of minimizers for $\tau_{\operatorname{min}}$ whose mass diverges to infinity.

\begin{proposition} \label{prpVertCont}
Let $|\mathcal{E}_{\operatorname{inf}}| >0$ and $l_{\H} <0$. 
    Assume that $\tau_{\operatorname{min}}$ admits a minimizer. Then, there exists $\mu_2 \geq \mu_{\operatorname{min}}$ such that $\tau_{\operatorname{min}}$ does not admit any minimizer of mass $\mu$ for any $\mu > \mu_2$.
\end{proposition}

\begin{proof} 
    Assume by contradiction that there exists $(\mu_n)_{n} \subseteq  \R^+$ such that $\mu_n \to \infty$ as $n \to \infty$ and, for all $n \in \mathbb N$, $\tau_{\operatorname{min}}$ admits a minimizer $\psi_n = (\psi_{n, e})_{e \in \mathcal{E}}$ of mass $\mu_n$.  Observe that $(\| \psi_n \|_{L^2(\mathcal K)})_{n}$ is bounded, by the same reasoning as in Remark \ref{rmkMConcentr}.
    Moreover, since $E_{\H} (\psi_n) = \tau_{\operatorname{min}}$, for every $n$, $\psi_n$ are global minimizers of the energy, and thus they satisfy \eqref{eqNls} with $\omega = 0$.
    
    For $p \geq 5$, by Proposition \ref{prpHLSol}, for any $e \in \mathcal{E}_{\operatorname{inf}}$, we have $\psi_{n, e} \equiv 0$. But as $\| \psi_n \|_{L^2 (\G)} \to \infty$ and $\| \psi_n \|_{L^2(\mathcal{K})} \lesssim 1$, this yields a contradiction.
    
    For $p \in (1, 5)$, assume that there exists $e \in \mathcal E_{\operatorname{inf}}$, identified with $[0, \infty)$, such that $\| \psi_{n, e} \|_{L^2(0, \infty)}^2 \to \infty$. By Proposition \ref{prpHLSol}, there exists $(a_n)_{n} \subseteq  \mathbb R^+$ such that $a_n \to 0$ and, for all $x \in [0, \infty)$,
    \begin{equation*}
        u_{n, e} (x)
        = \left( \frac{2(p + 1)}{(p - 1)^2} \right)^{\frac{1}{p - 1}} \left( x + a_n \right)^{-\frac{2}{p - 1}} .
    \end{equation*}
    Thus, for $\Llow$ defined in \eqref{eqMinLnght}, we have
    \[
        \| \psi_n \|_{L^2(\mathcal K)} \geq \| \psi_{n, e} \|_{L^2 (0, \Llow)} \to \infty,
    \]
    as $n \to \infty$, which yields a contradiction.
\end{proof}

Finally, we prove stability of the set $\mathcal{B}_\mu$, which concludes the proof of Theorem \ref{thmPSmall}.

\begin{proof}[Proof of Theorem \ref{thmPSmall}]
    The existence of energy ground states for small masses is
    given by Proposition \ref{prpMuOne}. Moreover, it is a direct consequence of  Corollary \ref{corFnlPlt1} and Proposition \ref{prpVertCont}
    that, for $p\in (1,5)$, $\tau_\mu$ does not admit any minimizer for sufficiently large masses.
    Finally, the claim regarding stability of energy ground states follows using
    the same proof as that of \cite[Theorem $\operatorname{II}.2.$]{CaLi82}.
\end{proof}

\subsection{Delta type vertex conditions}
\label{sec:delta}

In this section, we deal with the particular case where every vertex condition is given by a delta, which we now specify.  Let $(\alpha_v)_{v \in \mathcal V} \subseteq  \R$ and $\H_{\alpha}$ be the Hamiltonian operator with $\delta$-vertex conditions, i.e.
\begin{equation}\label{eqDefHAlpha}
    \mathcal{D}(\H_{\alpha})
    = \left\{ \psi \in H^2(\mathcal G)
    \Bigm| \text{for all } v \in \mathcal V, \, \psi \text{ is continuous at } v
    \text{ and } \sum_{e \succ v} \psi_e'(v) = \alpha_v \psi(v) \right\},
\end{equation}
where $\psi(v)$ denotes the common value at the vertex, where $e \succ v$ denotes
the edges $e \in \mathcal E$ adjacent to $v$ and $\alpha_v \in \R$ for all $v \in \mathcal V$. We highlight that in this case, writing $\psi(v)$ for the value of a function $\psi \in \mathcal{D} (\H_{\alpha})$ is allowed since all the functions in $D(\H_{\alpha})$ are continuous everywhere on the graph, and in particular at any vertex. 

The quadratic form  associated to $\H_{\alpha}$ is given by
\begin{equation*}
    Q_{\H_{\alpha}} (\psi)
    = \| \psi \|_{L^2(\G)}^2 + \sum_{v \in \mathcal V} \alpha_v |\psi(v)|^2
\end{equation*}
for $\psi$ in the domain
\begin{equation*}
    \mathcal{D}(Q_{\H_{\alpha}})
    = \left\{ \psi \in H^1(\G) \mid \text{for all } v \in \mathcal V, \, \psi \text{ is continuous at } v \right\},
\end{equation*}
see \cite[Chapter I]{BeKu13} for the details. To satisfy the condition $l_{\H_\alpha}<0$, some $\alpha_v$ have to be chosen negative.
A sufficient condition to obtain $l_{\H_\alpha}<0$ is the following.

\begin{remark}
    Assume that for some $v \in \mathcal V$, we have $\alpha_v < -d_v / \Llow$,
    with $\Llow$ given by \eqref{eqMinLnght}. Then, $l_{\H_{\alpha}} < 0$.
    Indeed, let $C = \sqrt{3/d_v \Llow}$ and $\psi = (\psi_e)$ such that
    \begin{equation*}
        \psi_e(x) =
        \begin{cases}
            -\frac{C}{\Llow}x + C, & x \in [0, \Llow] \\
            0,  & x \not \in [0, \Llow]
        \end{cases}
    \end{equation*}
    for any $e \in J_v$ identified with $[0,A_e]$ for some $A_e >0$ or $[0,\infty)$.
    Moreover, let $\psi_e \equiv 0$ for $e \not \in J_v$. Then $\psi \in \mathcal{D}(Q_{\H_{\alpha}})$ and
    \begin{equation*}
        \| \psi \|_{L^2(\G)} = 1, \quad \quad
        Q_{\H_{\alpha}} (\psi)
        = C^2 \left( \frac{d_v}{\Llow}+ \alpha_v \right) < 0,
    \end{equation*}
    implying that $l_{\H_\alpha}< 0$.
   
\end{remark}
For these types of vertex conditions, we are able to obtain more information about the existence of ground states.
First, we will show that in this case, all the ground states are strictly positive up to a constant phase shift. 

\begin{lemma} \label{lemPos}
    Let $|\mathcal{E}_{\operatorname{inf}}| > 0$ and $l_{\H_\alpha} <0$.
    Let $\mu >0$, and let $u$ be a minimizer for $\tau_\mu$. Then there exists $\theta \in \R$ such that $e^{i\theta}u >0$.
\end{lemma}

\begin{proof}
Let $u$ be any minimizer for $\tau_\mu$ for some $\mu > 0$, and set $\psi = |u|$.  
Then $\psi$ is also a minimizer and satisfies \eqref{eqNls} for some $\omega \in \R$.  
By Proposition \ref{prpPropert}, we have $\psi \in C^2$ in the interior of each edge $e \in \mathcal{E}$.  
Let $\psi = (\psi_e)_{e \in \mathcal{E}}$. We now show that $\psi > 0$ on $\G$.

\medskip

\noindent \textit{Step 1: Positivity at vertices.}
Assume, by contradiction, that there exists a vertex $y \in \mathcal V$ such that $\psi(y) = 0$.  
Since $\psi \in \mathcal{D}(\H_\alpha)$, we have
\[
\sum_{e \in J_y} \psi_e'(y) = 0.
\]  
Then either
\begin{enumerate}
    \item there exist edges $e_1, e_2 \succ y$ such that $\psi_{e_1}'(y) > 0$ and $\psi_{e_2}'(y) < 0$, or
    \item $\psi_e'(y) = 0$ for all $e \in J_y$.
\end{enumerate}
In the first case, continuity implies $\psi_{e_2} < 0$ near $y$, contradicting $\psi \ge 0$.  
In the second case, by uniqueness of solutions to \eqref{eqNls2} on each edge, we get $\psi_e \equiv 0$ for all $e \in J_y$.  

If $e \in \mathcal{E}_{\operatorname{fin}}$, identified with $I_e = [0,L_e]$ and $y = 0$, then by continuity $\psi_e(L_e) = 0$, so the argument can be iterated along the graph.  
If $e$ is infinite, identified with $[0,\infty)$, we also get $\psi_e \equiv 0$.  
In either case, this leads to $\psi \equiv 0$ on $\G$, which contradicts the fact that $E_{\H_\alpha}(\psi) = \tau_\mu < 0$.

\medskip

\noindent \textit{Step 2: Positivity on edges.}  
Suppose, by contradiction, that there exists an edge $e \in \mathcal{E}$, identified with $I_e \subset \R$, and $x \in I_e$ such that $\psi_e(x) = 0$.  
Since $\psi \ge 0$, $x$ is a global minimum, and thus $\psi_e'(x) = 0$.  
By uniqueness of solutions to \eqref{eqNls2} on $I_e$, it follows that $\psi_e \equiv 0$ on $I_e$.  
By continuity, this implies that $\psi$ vanishes at a vertex of $\G$, which contradicts Step 1. Combining Steps 1 and 2, we conclude that $\psi > 0$ on all of $\G$. 

\medskip

\noindent \textit{Step 3: Constant complex phase.}  
Finally, we show that $u$ is strictly positive up to a complex phase shift. As we have $\rho(x) = |u(x)|>0$, for any minimizer $u = (u_e)_{e\in \mathcal{E}}$, we can write $u_e(x) = \rho_e(x) e^{i\theta_e(x)}$ for $\rho_e,\theta_e$ real-valued by the one dimensional Lifting theorem.

We note that $\rho_e \in \mathcal{D} (\H)$, as for any vertex, identified with zero, the condition
\begin{equation*}
    \sum_{e \succ v} e^{-i\theta_e(0)} u_e'(0) =  \sum_{e \succ v} (\rho_e'(0) + i \rho_e(0)\theta_e'(0)) = \alpha_v \rho_e(0),
\end{equation*}
implies, separating the real and imaginary parts, that
\begin{equation*}
    \sum_{e \succ v} \rho_e'(0) = \alpha_v \rho_e(0),
    \quad
    \sum_{e \succ v} \theta_e'(0) = 0.
\end{equation*}
On the other hand, if $\theta_e' \not \equiv 0$ for some $e$ then 
\begin{equation*}
    \| u_e'\|_{L^2(e)}^2 =  \| \rho_e'\|_{L^2(e)}^2  +  \| \rho_e \theta_e' \|_{L^2(e)}^2  > \| \rho_e'\|_{L^2(e)}^2,
\end{equation*}
which yields 
\begin{equation*}
    E_{\H}(u) >  E_{\H}(\rho), \quad \| u \|_{L^2(\G)} = \| \rho \|_{L^2(\G)},
\end{equation*}
which contradicts the fact that $u$ is an energy minimizer. As a consequence, $\theta_e(x) = \theta_e \in \R$ is fixed on any edge. By the continuity on each vertex, we conclude that the phase shift is constant on the whole graph. 
\end{proof}

Thanks to Lemma \ref{lemPos} and Proposition \ref{prpPropert}, we can deduce
a complete picture of the existence of ground states in the case $p \geq 5$ as stated in the next Proposition.

\begin{proposition} \label{prpSupCrit}
    Let $|\mathcal{E}_{\operatorname{inf}}| >0$, $l_{\H_\alpha} <0$ and $p \geq 5$. Then for all $\mu > 0$, there exists a minimizer for $\tau_\mu$.
\end{proposition}

\begin{proof}
    Assume by contradiction that there exists $\hat \mu > 0$ such that $\tau_{\hat{\mu}}$ does not admit a minimizer.  Then, there exists $0 < \bar{\mu} < \hat{\mu}$ such that $\mu \mapsto \tau_{\mu}$ is constant on the interval $[\bar{\mu}, \hat{\mu}]$ and $\tau_{\bar{\mu}} = \tau_{\hat{\mu}}$ admits a non-negative minimizer $u$ of mass $\bar{\mu}$. Notice that by Proposition \ref{prpMuOne}, $\bar{\mu} >0$. By Remark \ref{rmkOmegaZero}, $u$ satisfies \eqref{eqNls} for a Lagrange multiplier $\omega \leq 0$. In particular, $u_e \equiv 0$ for all $e \in \mathcal E_{\operatorname{inf}}$ by Proposition \ref{prpHLSol}. This contradicts Lemma \ref{lemPos}.
\end{proof}

Moreover, we also obtain a complete picture in the case $p \in (1, 5)$, for graphs with one vertex. Indeed, in this case, independently of the structure of the graph, we show that minimizers for $\tau_\mu$ exist for any $\mu \in [0,\mu_{\operatorname{min}}]$ and do not exist for any $\mu > \mu_{\operatorname{min}}$, where $\mu_{\operatorname{min}}$ is defined in \eqref{eqMinMass}.  

\begin{proposition} \label{prpUniVer}
    Let $|\mathcal{E}_{\operatorname{inf}}| > 0$, $|\mathcal{V}| = 1$, $l_{\H_\alpha} < 0$ and $p \in (1, 5)$. Then $\mu \to \tau_\mu$ is strictly decreasing in $[0, \mu_{\operatorname{min}}]$ and constant for $\mu > \mu_{\operatorname{min}}$. Moreover, $\tau_{\operatorname{min}}$ does not admit any minimizer with mass $\mu > \mu_{\operatorname{min}}$.
\end{proposition}

To prove this proposition, we use the following lemma. 

\begin{lemma} \label{lemMinMax}
    Let $|\mathcal{E}_{\operatorname{inf}}| >0$.
    Suppose \eqref{eqNls} with $\H = \H_\alpha$ for some $\alpha$ admits two positive solutions $u$, $\hat{u}$ for the same $\omega \geq 0$. If $\max(u - \hat{u})$ or $\min(u - \hat{u})$ are attained in $\mathring{e}$ for some $e \in \mathcal E$, then $u \equiv \hat{u}$.
\end{lemma}

\begin{proof}
    Let $u$, $\hat{u}$ be two positive solutions to \eqref{eqNls} with the same $\omega \geq 0$. 
    By the self-adjointness of $\H_\alpha$, we have 
    \begin{equation*}
        \langle \H_{\alpha} u, \hat{u} \rangle_{L^2(\G)} = \langle  u, \H_{\alpha} \hat{u} \rangle_{L^2(\G)}
    \end{equation*}
    Using \eqref{eqNls}, this yields
    \begin{equation} \label{eqUniVer1}
        \int_{\mathcal G} u (x) \hat u (x) \left( u (x)^{p-1} - \hat u (x)^{p-1} \right) \, dx= 0.
    \end{equation}
Suppose that $x_0 \in \operatorname{argmax}(u - \hat{u})$ is such that $x_0 \in\mathring{e}$ for some edge $e \in \mathcal{E}$. Recall that $u,\hat{u} \in C^2(e)$ by Proposition \ref{prpPropert}. Then, by \eqref{eqNls2} we have 
\begin{equation}\label{eqCalc1}
\begin{aligned}
     0 \geq & \left( u_e - \hat{u}_e \right)''(x_0)
        = \omega\left(  u_e(x_0) - \hat{u}_e (x_0)  \right) + \left( u_e (x_0)^p - \hat{u}_e (x_0)^p \right) \\
        & = (u_e(x_0)-\hat{u}_e (x_0))
        \left( \omega + p \int_0^1 \bigl((1 -s) \hat{u}_e (x_0) + su_e(x_0)\bigr)^{p-1} \, ds \right),
\end{aligned}
\end{equation}
    Since $\omega \geq 0$ and, by hypothesis, $u_e(x_0), \hat{u}_e(x_0) \geq 0$, \eqref{eqCalc1} implies
\[
u_e(x_0) - \hat{u}_e(x_0) \leq 0.
\]
In particular, by the definition of $x_0$, it follows that for all $x \in \G$,
\[
u(x) - \hat{u}(x) \leq u_e(x_0) - \hat{u}_e(x_0) \leq 0,
\]
that is,
\[
0 \leq u(x) \leq \hat{u}(x).
\]
Then, \eqref{eqUniVer1} implies $u \equiv \hat{u}$, which concludes the first part of the proof.  

To prove the same property for $\min(u - \hat{u})$, we observe that the only difference lies in the opposite inequality in \eqref{eqCalc1}; hence, the argument proceeds analogously.
\end{proof}

\begin{proof}[Proof of Proposition \ref{prpUniVer}]
    By Proposition \ref{prpMuOne}, $\mu \to \tau_\mu$ is decreasing in an interval $[0,\mu_1]$, and $\mu_1 \leq \mu_{\operatorname{min}}$ where $\mu_{\operatorname{min}}$ is defined in \eqref{eqMinMass}. Moreover, by Corollary \ref{corFnlPlt1} and Lemma \ref{lemPos}, there exists a strictly positive minimizer $u_{\operatorname{min}}$ for the problem $\tau_{\operatorname{min}}$ such that $\| u_{\operatorname{min}}\|_{L^2(\G)} = \mu_{\operatorname{min}}$. Since it is a global minimizer of the energy, it satisfies \eqref{eqNls} with $\H = \H_\alpha$ for $\omega = 0$.

    \medskip
    
    \noindent \textit{Step 1: $\mu_1 = \mu_{\operatorname{min}}$.} Suppose by contradiction that $\mu_1 < \mu_{\operatorname{min}}$. So there exists $\mu_2 > \mu_1$ such that $\mu \mapsto \tau_\mu$ is constant on $[\mu_1,\mu_2]$. By Lemma \ref{lemExistMin}, Remark \ref{rmkOmegaZero}, and Lemma \ref{lemPos}, there exists a minimizer for $\tau_{\mu_1}$, all the minimizers for $\tau_{\mu_1}$ are strictly positive up to a constant phase shift, and they satisfy \eqref{eqNls} with $\H = \H_\alpha$ for some $\omega \leq 0$. By Proposition \ref{prpHLSol},
    it follows that they satisfy \eqref{eqNls} with $\omega = 0$. Let $u >0$ be one of these minimizers. Then by Lemma \ref{lemMinMax} we have that either $u = u_{\operatorname{min}}$ which is absurd as $\| u \|_{L^2(\G)} = \mu_1 < \mu_{\operatorname{min}}$ or both $\max(u - u_{\operatorname{min}})$ and  $\min(u - u_{\operatorname{min}})$ are attained in a vertex of $\G$. But since $|\mathcal{V}| =1$, it again yields that $u = u_{\operatorname{min}}$, contradicting the assumptions. Thus, $\mu_1 = \mu_{\operatorname{min}}$.

    \medskip

    \noindent \textit{Step 2: $\tau_{\operatorname{min}}$ does not admit any minimizer for any $\mu > \mu_{\operatorname{min}}$.} Indeed, assume by contradiction that there exists one minimizer $u$ such that $\| u \|_{L^2(\G)} = \mu$. Then by Lemma \ref{lemPos}, it can be chosen strictly positive. As it is a global minimizer of the energy, it satisfies \eqref{eqNls} with $\omega = 0$. Thus, we can repeat the argument of \textit{step 1} to reach a contradiction. 
\end{proof}

\subsection{Bifurcation from the linear ground state}

In this subsection, we pursue our study of
$\delta$-type vertex conditions. In this setting, a straightforward adaptation of Lemma~\ref{lemPos} shows that the lowest eigenvalue $l_{\H_\alpha}$ is simple as the associated normalized eigenfunction $\phi_0$ is strictly positive (up to a complex phase shift).

As a consequence, for frequencies $\omega$ sufficiently close to $-l_{\H_\alpha}$, the nonlinear ground states can be constructed as bifurcations from the linear ground state. This bifurcation analysis follows the approach introduced in~\cite{KiKePe11}, which relies on a Lyapunov--Schmidt reduction around the simple eigenvalue $l_{\H_\alpha}$.

\begin{proposition}\label{prpBifurcation}
Let $|\mathcal{E}_{\operatorname{inf}}|>0$ and $l_{\H_\alpha}<0$.
Then there exists $\delta>0$ such that, for any
$\omega \in (-l_{\H_\alpha}-\delta,-l_{\H_\alpha})$,
there exists a unique real-valued solution $\phi(\omega)$
to~\eqref{eqNls}, up to a change of sign. \footnote{Since by
Lemma~\ref{lemPos} every ground state is real-valued up to a phase shift,
one may easily apply this statement of energy ground states for small $\mu$.}
It is of the form
\begin{equation}\label{eqDecompos}
\phi(\omega)=a\phi_0+\Theta,
\qquad
\langle \phi_0,\Theta\rangle=0,
\qquad
\|\Theta\|_{H^1}\lesssim a^{p}.
\end{equation}
Moreover, the function
$m(\omega)=\|\phi(\omega)\|_{L^2(\G)}^2$
satisfies
\begin{equation*}
m(\omega)
=
\left(
-\frac{\omega+l_{\H_\alpha}}
{\|\phi_0\|_{L^{p+1}(\G)}^{p+1}}
\right)^{\frac{2}{p-1}}
+
O\Bigl( (-\omega - l_{\H_{\alpha}})^{\frac{4}{p-1}} \Bigr).
\end{equation*}
Furthermore, $m\in C^1(-l_{\H_\alpha}-\delta,-l_{\H_\alpha})$
and is invertible on this interval.
\end{proposition}

\begin{proof}
In all this proof, we restrict the spaces to real-valued functions
    (without changing notations for simplicity).
We define the map
\[
F : \mathcal{D} (\H_\alpha)\times [0,-l_{\H_\alpha}] \to L^2(\G)
\]
by
\begin{equation*}
F(\phi,\omega)
=
(\H_\alpha+\omega)\phi+|\phi|^{p-1}\phi .
\end{equation*}

\medskip

\noindent \textit{Step 1: Existence of $\phi(\omega)$.}
Notice that
\[
D_\phi F(\phi,\omega)
=
(\H_\alpha+\omega)+p|\phi|^{p-1},
\]
and therefore
\[
D_\phi F(0,\omega)=\H_\alpha+\omega.
\]
Let $P : D(Q_{\H_\alpha}) \to D(Q_{\H_\alpha})$ be defined by
\[
P(\cdot)=\phi_0\langle \phi_0,\cdot\rangle_{L^2 (\G)},
\]
and set $Q=I-P$, where $\phi_0$ is the normalized positive eigenfunction
associated with the eigenvalue $l_{\H_\alpha}$.
We look for zeros of $F$ of the form~\eqref{eqDecompos}, where in particular
$a=\langle \phi(\omega),\phi_0\rangle_{L^2 (\G)}$.
Accordingly, we redefine $F$ by
\[
F(a,\Theta,\omega)
=
F(a\phi_0+\Theta,\omega).
\]
Notice that if $\phi$ is a solution to~\eqref{eqNls}, then so is $-\phi$.
Hence, without loss of generality, we may assume $a\geq 0$.

   Let us define $G$ as the projection of $F$ onto the orthogonal complement
of the direction generated by $\phi_0$, namely
\[
G(a,\Theta,\omega)=QF(a,\Theta,\omega).
\]
After straightforward computations, exploiting the self-adjointness
of $\H_\alpha$, the orthogonality condition
$\langle \Theta,\phi_0\rangle_{L^2 (\G)}=0$, and the fact that $Q\Theta=\Theta$,
one finds that $G$ can be written as
\begin{equation*}
G(a,\Theta,\omega)
=
Q(\H_\alpha+\omega)Q\Theta
+
Q\bigl(|a\phi_0+\Theta|^{p-1}(a\phi_0+\Theta)\bigr).
\end{equation*}
Now, since
\begin{equation*}
D_\Theta G(0,0,-l_{\H_\alpha})
=
Q(\H_\alpha-l_{\H_\alpha})Q
\end{equation*}
is invertible, as $Q$ projects outside the simple kernel of
$(\H_\alpha-l_{\H_\alpha})$, we may apply the implicit function theorem.
As a consequence, there exist $\varepsilon,\eta > 0$ and a set
\[
I=[0,\varepsilon)\times(-l_{\H_\alpha}-\eta,-l_{\H_\alpha}+\eta)
\]
such that for any $(a,\omega)\in I$ there exists a unique function
$\Theta\in C^1(I,D(\H_\alpha))$ satisfying
\begin{equation}\label{eqGZero}
G(a,\Theta(a,\omega),\omega)=0 \text{ for all } (a,\omega)\in I.
\end{equation}
From~\eqref{eqGZero} and the invertibility of the linear operator
$Q(\H_\alpha+\omega)Q$ for $\omega$ sufficiently close to $-l_{\H_\alpha}$,
we obtain
\begin{equation*}
\Theta(a,\omega)
=
\bigl(Q(\H_\alpha+\omega)Q\bigr)^{-1}
\bigl(
Q|a\phi_0+\Theta|^{p-1}(a\phi_0+\Theta)
\bigr).
\end{equation*}
Since the inverse $\bigl(Q(\H_\alpha+\omega)Q\bigr)^{-1}$ is bounded
for $\omega$ sufficiently close to $-l_{\H_\alpha}$, it follows that
\begin{equation*}
\|\Theta(a,\omega)\|_{H^1(\G)}
\lesssim
\|a\phi_0\|_{H^1(\G)}^{p}
+
\|\Theta(a,\omega)\|_{H^1(\G)}^{p}.
\end{equation*}
By possibly shrinking the set $I$, the second term on the right-hand side
can be absorbed into the left-hand side, yielding
\begin{equation}\label{eqThetaControl}
\|\Theta(a,\omega)\|_{H^1(\G)}\lesssim a^{p}.
\end{equation}
Differentiating~\eqref{eqGZero} with respect to $\omega$, we obtain
\begin{equation*}
Q\Theta(a,\omega)
+
Q(\H_\alpha+\omega)Q\,\partial_\omega\Theta(a,\omega)
+
p\,Q\Bigl(
|a\phi_0+\Theta(a,\omega)|^{p-1}
\partial_\omega\Theta(a,\omega)
\Bigr)
=0.
\end{equation*}
Arguing as above, this implies
\begin{align*}
\|\partial_\omega\Theta(a,\omega)\|_{H^1(\G)}
\lesssim\,
&\|\Theta(a,\omega)\|_{H^1(\G)} \\
&+
\Bigl(
\|a\phi_0\|_{H^1(\G)}^{p-1}
+
\|\Theta(a,\omega)\|_{H^1(\G)}^{p-1}
\Bigr)
\|\partial_\omega\Theta(a,\omega)\|_{H^1(\G)}.
\end{align*}
Using~\eqref{eqThetaControl} and possibly restricting $I$ further, the second term on the right-hand side can again be absorbed into the
left-hand side, which yields
\begin{equation}\label{eqThetaControl1}
\|\partial_\omega\Theta(a,\omega)\|_{H^1(\G)}
\lesssim
\|\Theta(a,\omega)\|_{H^1(\G)}
\lesssim a^{p}.
\end{equation}
Similarly, differentiating~\eqref{eqGZero} with respect to $a$, we find
\begin{equation*}
    Q(\H_\alpha+\omega)Q\,\partial_a\Theta(a,\omega)
    +
    p\,Q\Bigl(
    |a\phi_0+\Theta(a,\omega)|^{p-1}
    (\phi_0 + \partial_a\Theta(a,\omega))
    \Bigr)
    =0.
\end{equation*}
Isolating the term involving $\partial_a \Theta$ and utilizing the invertibility of $Q(\H_\alpha+\omega)Q$, we obtain the estimate
\begin{align*}
    \|\partial_a\Theta(a,\omega)\|_{H^1(\G)}
    \lesssim\,
    &\| |a\phi_0+\Theta(a,\omega)|^{p-1} \phi_0 \|_{L^2(\G)} \\
    &+
    \Bigl(
    \|a\phi_0\|_{H^1(\G)}^{p-1}
    +
    \|\Theta(a,\omega)\|_{H^1(\G)}^{p-1}
    \Bigr)
    \|\partial_a\Theta(a,\omega)\|_{H^1(\G)}.
\end{align*}
The first term on the right-hand side is bounded by $O(a^{p-1})$. Absorbing the second term into the left-hand side for small $a$, we conclude
\begin{equation}\label{eqThetaControl2}
    \|\partial_a\Theta(a,\omega)\|_{H^1(\G)} \lesssim a^{p-1}.
\end{equation}

Next, we consider the complementary problem of finding a root of 
$P F(a, \Theta(a,\omega), \omega)$ for $(a,\omega)\in I$. 
Using the definition of $P$ and the orthogonality condition 
$\langle \Theta(a,\omega), \phi_0\rangle_{L^2{\G}} = 0$, one immediately obtains
\begin{equation*}
P F(a, \Theta(a,\omega), \omega) 
=
\phi_0 \Biggl(
a(l_{\H_\alpha} + \omega) 
+
\bigl\langle \phi_0, |a\phi_0 + \Theta(a,\omega)|^{p-1} (a\phi_0 + \Theta(a,\omega)) \bigr\rangle_{L^2 (\G)}
\Biggr).
\end{equation*}
We define the auxiliary function
\begin{equation}\label{eqAuxF}
f(a,\omega)
=
(l_{\H_\alpha} + \omega)
+
\left\langle \phi_0, 
\left| a\phi_0 + \Theta(a,\omega) \right|^{p-1} 
\left( \phi_0 + \frac{\Theta(a,\omega)}{a} \right)
\right\rangle_{L^2 (\G)}.
\end{equation}
From \eqref{eqThetaControl}, it follows that $f$ is continuous on $I$. 
Moreover, 
\begin{equation*}
\partial_\omega f(a,\omega) 
=
1 - p \left\langle \phi_0, |a\phi_0 + \Theta(a,\omega)|^{p-1} \frac{\partial_\omega \Theta(a,\omega)}{a} \right\rangle_{L^2 (\G)}.
\end{equation*}
Exploiting \eqref{eqThetaControl} and \eqref{eqThetaControl1}, we see that 
$\partial_\omega f(a,\omega)$ is also continuous. 
Since $\partial_\omega f(0,l_{\H_\alpha}) = 1$, the implicit function theorem 
guarantees the existence of a uniquely defined $\omega(a)$ for 
$a \in [0, \eps_0)$ for some $\eps_0 \in (0,\eps)$ such that
\begin{equation}\label{eqAuxF1}
f(a, \omega(a)) = 0 \text{ for all } a \in (0, \eps_0).
\end{equation}
In particular, this also implies
\begin{equation}\label{eqPFzero}
P F(a, \Theta(a, \omega(a)), \omega(a)) = 0 \text{ for all } a \in (0, \eps_0).
\end{equation}
Combining \eqref{eqGZero} and \eqref{eqPFzero}, we conclude that for $a \in (0,\eps_0)$,
\begin{equation*}
F(a, \Theta(a,\omega(a)), \omega(a)) = 0.
\end{equation*}

\medskip

\noindent \textit{Step 2: Inverting $\omega(a)$.} 
Since $f(a, \omega(a)) = 0$ for all $a \in (0, \varepsilon_0)$, differentiating with respect to $a$ yields
\begin{equation*}
    \omega'(a) = - \frac{\partial_a f(a, \omega(a))}{\partial_\omega f(a, \omega(a))}.
\end{equation*}
The partial derivative with respect to $a$ is given by
\begin{align*}
    \partial_a f(a,\omega) 
    &= \frac{1}{a} \left\langle \phi_0, p |a\phi_0 + \Theta|^{p-1} \left(\phi_0 + \partial_a \Theta\right) \right\rangle_{L^2(\G)} 
    - \frac{1}{a^2} \left\langle \phi_0, |a\phi_0 + \Theta|^{p-1} (a\phi_0 + \Theta) \right\rangle_{L^2(\G)} \\
    &= (p-1) a^{p-2} \|\phi_0\|_{L^{p+1}(\G)}^{p+1} + O(a^{2p-3}),
\end{align*}
where the estimates on the remainder follow
from \eqref{eqThetaControl} and \eqref{eqThetaControl2}.
Moreover, from \textit{step 1} we have $\partial_\omega f(a, \omega) = 1 + O(a^{p-1})$.
Combining these estimates, we obtain
\begin{equation*}
    \omega'(a) = - (p-1) a^{p-2} \|\phi_0\|_{L^{p+1}(\G)}^{p+1} + O(a^{2p-3}).
\end{equation*}
Thus, $\omega(a)$ is strictly decreasing and therefore invertible on $(0, \eps_0)$.

To obtain the expansion of the inverse map, we first integrate $\omega'(a)$ using the condition $\omega(0) = -l_{\H_{\alpha}}$. This yields
\begin{equation*}
    \omega(a) = -l_{\H_{\alpha}} - a^{p-1} \|\phi_0\|_{L^{p+1}(\G)}^{p+1} + O(a^{2p-2}).
\end{equation*}
Let $J \coloneqq (- l_{\H_{\alpha}} - \delta, - l_{\H_{\alpha}})$ for sufficiently small $\delta>0$,
on which $\tilde a : J \to \R$ is the inverse of $\omega(a)$.
For $\omega \in J$, we invert the relation above, obtaining
\begin{equation*}
    \tilde{a}(\omega)^{p-1}
    = \frac{-\omega - l_{\H_{\alpha}}}{\|\phi_0\|_{L^{p+1}(\G)}^{p+1}} + O(a^{2p-2}).
\end{equation*}
Taking the root and observing that
$a \approx (-\omega - l_{\H_{\alpha}})^{\frac{1}{p-1}}$, we obtain
\begin{equation}\label{eqAOmega}
    \tilde{a}(\omega) 
    = \left( \frac{-\omega - l_{\H_{\alpha}}}{\|\phi_0\|_{L^{p+1}(\G)}^{p+1}} \right)^{\frac{1}{p-1}} 
    + O\Bigl( (-\omega - l_{\H_{\alpha}})^{\frac{2}{p-1}} \Bigr),
\end{equation}
for $\omega \in J$. This concludes the second step.

\medskip
\noindent \textit{Step 3: Computation of the mass.} 
It remains to analyze the dependence of the mass on the frequency. 
We first express the mass as a function $M$ of the parameter $a$, given by
\begin{equation*}
    M(a) = \| a\phi_0 + \Theta(a, \omega(a)) \|_{L^2(\G)}^2 = a^2 + \| \Theta(a, \omega(a)) \|_{L^2(\G)}^2.
\end{equation*}
Differentiating with respect to $a$, we obtain
\begin{equation*}
    M'(a) = 2a + 2 \left\langle \Theta(a, \omega(a)), \frac{d}{da} \Theta(a, \omega(a)) \right\rangle_{L^2(\G)}.
\end{equation*}
By the chain rule, the total derivative of $\Theta$ is
\begin{equation*}
    \frac{d}{da} \Theta(a, \omega(a)) = \partial_a \Theta(a, \omega(a)) + \partial_\omega \Theta(a, \omega(a)) \, \omega'(a).
\end{equation*}
We now estimate the remainder term using the bounds established in the previous steps: 
$\|\Theta\|_{L^2} \lesssim a^p$, $\|\partial_a \Theta\|_{L^2} \lesssim a^{p-1}$, 
$\|\partial_\omega \Theta\|_{L^2} \lesssim a^p$, and $|\omega'(a)| \lesssim a^{p-2}$. 
This yields
\begin{equation*}
    \left\| \frac{d}{da} \Theta(a, \omega(a)) \right\|_{L^2(\G)} 
    \le \|\partial_a \Theta\|_{L^2(\G)} + \|\partial_\omega \Theta\|_{L^2(\G)} |\omega'(a)|
    \lesssim a^{p-1} + a^p \cdot a^{p-2} \lesssim a^{p-1}.
\end{equation*}
Hence, the inner product in the expression of $M'(a)$ is bounded by
\begin{equation*}
    \left| \left\langle \Theta, \frac{d\Theta}{da} \right\rangle_{L^2 (\G)} \right| 
    \le \|\Theta\|_{L^2(\G)} \left\| \frac{d\Theta}{da} \right\|_{L^2(\G)} 
    \lesssim a^p \cdot a^{p-1} = a^{2p-1}.
\end{equation*}
Since $p > 1$, we have $2p-1 > 1$, and thus
\begin{equation*}
    M'(a) = 2a + O(a^{2p-1}) = 2a(1 + o(1)).
\end{equation*}
For $a$ sufficiently small, $M'(a) > 0$. 
Since we established in Step 2 that $\omega'(a) < 0$ for small $a$, the derivative of the mass with respect to frequency is
\begin{equation*}
    \frac{d}{d\omega} m(\omega) = \frac{M'(\tilde a(\omega))}{\omega'(\tilde a(\omega))} < 0.
\end{equation*}
This proves that the map $\omega \mapsto m(\omega)$ is strictly decreasing and invertible near $-l_{\H_\alpha}$.
Finally, substituting the expansion \eqref{eqAOmega} into the expansion of the mass, we obtain
\begin{equation*}
    m(\omega) = \tilde a(\omega)^2 + O(\tilde a(\omega)^{2p}) 
    = \left(- \frac{\omega + l_{\H_{\alpha}}}{\|\phi_0\|_{L^{p+1}(\G)}^{p+1}} \right)^{\frac{2}{p-1}} 
    + O\Bigl( (-\omega - l_{\H_{\alpha}})^{\frac{4}{p-1}} \Bigr).
    \qedhere
\end{equation*}
\end{proof}

\begin{remark}
    The same proof applies to any simple negative eigenvalue of $\H_\alpha$. 
    However, guaranteeing the simplicity of eigenvalues for a generic graph, even for $\delta$-type vertex conditions, is a delicate issue that lies beyond the scope of the present work. 
    On the other hand, for every negative eigenvalue, the existence of the corresponding nonlinear stationary states will be established in the next section.
\end{remark}

We can finally prove Theorem \ref{thmDelta}.

\begin{proof}[Proof of Theorem \ref{thmDelta}]
    The proof follows from Lemma \ref{lemPos} and Propositions \ref{prpSupCrit}, \ref{prpUniVer}, and \ref{prpBifurcation}.
\end{proof}

\section{Multiplicity results}
\label{secMult}

\subsection{Solutions with a prescribed frequency}

We first establish that the action satisfies the Palais-Smale condition.

\begin{proposition}[Palais-Smale condition for $S_\omega$] \label{PSaction}
    Let $\omega > 0$. The functional $S_\omega$ satisfies the Palais-Smale condition at any level $c \in \mathbb{R}$.
    Namely, any sequence $(\psi_n)_n \subseteq \DQH$ with
    \begin{equation*}
        S_\omega(\psi_n) \to c \quad \text{and} \quad S_\omega'(\psi_n) \to 0 \text{ in } \DQH'
    \end{equation*}
    admits a strongly convergent subsequence.
\end{proposition}

\begin{proof}
    Since $\omega > 0$
    and since $(S_\omega(\psi_n))_n$ is bounded,
    we deduce that $(E_{\H}(\psi_n))_n$ is bounded
    (recall that $E_{\H}$ is bounded from below,
    according to Lemma~\ref{lemMinGlob}).
    We then deduce, as $\omega > 0$,
    that $(\psi_n)_n$ is bounded in $L^2(\G)$,
    so that Proposition~\ref{prpBddEnergy}
    implies that it is bounded in $H^1(\G)$.
    Hence, up to a subsequence, there exists $\psi \in H^1(\G)$
    such that $\psi_n \rightharpoonup \psi$ weakly in $H^1(\G)$.
    Furthermore, $S_{\omega}'(\psi) = 0$ as weak limits of Palais-Smale sequences are critical points.

    It remains to show that $(\psi_n)_n$ converges strongly to $\psi$ in $\DQH$.
    We first remark that
    \begin{multline}
        \langle S_{\omega}'(\psi_n) - S_{\omega}'(\psi), \psi_n - \psi \rangle_{L^2 (\G)}
        = \langle \H (\psi_n - \psi), \psi_n - \psi \rangle_{L^2 (\G)} \\
        + \langle |\varphi_n|^{p-1} \varphi_n
        - |\varphi|^{p-1} \varphi, \psi_n - \psi \rangle_{L^2 (\G)}.
    \end{multline}
    As $n \to \infty$, the left-hand side converges to $0$.
    Writing $\psi_n = \psi_n^+ + \psi_n^-$
    with $\psi_n^- \in E_-$ and $\psi_n^+ \in (E_-)^{\perp}$,
    we obtain
    \begin{multline*}
        o (1)
        = \langle \H (\psi_n^+ - \psi^+), \psi_n^+ - \psi^+ \rangle_{L^2 (\G)}
        + \langle \H (\psi_n^- - \psi^-), \psi_n^- - \psi^- \rangle_{L^2 (\G)} \\
        + \int_{\G} (|\psi_n (x)|^{p-1} \psi_n (x) - |\psi (x)|^{p-1} \psi (x))(\psi_n (x) - \psi (x)) \, dx.
    \end{multline*}
    Now, we remark that:
    \begin{itemize}
        \item since $\psi_n \rightharpoonup \psi$, we have $\psi_n^- \to \psi^-$ strongly
        as $E_-$ is finite dimensional;
        thus, $\langle \H (\psi_n^- - \psi^-), \psi_n^- - \psi^- \rangle_{L^2 (\G)} \to 0$;
        \item the nonlinear term is nonnegative
        due to the monotonicity of $s \mapsto |s|^{p-1}s$;
        \item the term involving $\psi_n^+ - \psi^+$ is also nonnegative.
    \end{itemize}
    
    From those facts, we deduce that all terms individually converge to zero. Hence,
    \begin{equation*}
        \langle \H (\psi_n^+ - \psi^+), \psi_n^+ - \psi^+ \rangle_{L^2 (\G)} \to 0
    \end{equation*}
    and
    \begin{equation*}
        \langle \H (\psi_n^- - \psi^-), \psi_n^- - \psi^- \rangle_{L^2 (\G)} \to 0
    \end{equation*}
    so that
    \begin{equation*}
        \langle \H (\psi_n - \psi), \psi_n - \psi \rangle_{L^2 (\G)} \to 0,
    \end{equation*}
    showing that $\psi_n$ converges strongly to $\psi$ in $\DQH$.
\end{proof}

\begin{proof}[Proof of Theorem~\ref{thmMultAction}]
    We first prove the result in the case where
    the matrices $\Lambda_v$ are real symmetric
    and restricting the action functional
    $S_\omega$ to the real subspace
    $X_{\R} \coloneqq \DQH \cap H^1(\G, \R)$.
    We apply Lusternik-Schnirelmann theory
    to the functional $S_\omega \in C^1(X_{\R}, \R)$,
    noting that it is even. In this case, we will show
    the existence of $k$ pairs of real-valued solutions
    $\pm u_1, \dotsc, \pm u_m$.
    
    We define the family of symmetric subsets
    of genus at least $j$ by
    \begin{equation*}
        \Sigma_j
        \coloneqq \bigl\{ A \subseteq \DQH \setminus \{0\}
        \mid  A \text{ is symmetric, compact, and }
        \gamma(A) \ge j \bigr\},
    \end{equation*}
    where the Krasnoselskii \textit{genus} $\gamma(A)$ of a nonempty symmetric compact set $A$ is defined
    as the smallest integer $k \ge 1$ such that there exists
    a continuous odd map $h : A \to \R^k \setminus \{0\}$.
    If no such $k$ exists, we define $\gamma(A) = \infty$.
    Finally, by convention, $\gamma(\emptyset) = 0$.

    We then define the minimax levels for $1 \le j \le k$ by
    \begin{equation}
        \label{def_cj_action}
        c_j \coloneqq \inf_{A \in \Sigma_j} \sup_{u \in A} S_{\omega}(u).
    \end{equation}
    
    \medskip
    
    \noindent \textit{Step 1:
    The levels are well-defined and negative.} Since $S_{\omega}$ is bounded from below,
    all $c_j$ are finite.
    Next, we verify that $c_j < 0$ for all $1 \le j \le k$.
    For any $\psi \in E_-$ with $\|\psi\|_{L^2 (\G)} = 1$, we have
    \begin{equation*}
        \langle \H\psi, \psi \rangle_{L^2 (\G)} + \omega \|\psi\|_{L^2 (\G)}^2
        \le \lambda_k + \omega
        < 0,
    \end{equation*}
    since $\omega < -\lambda_k$.
    Then, let us consider the sphere
    \begin{equation*}
        \mathcal{M}_\rho \coloneqq \{ \psi \in E_- \mid \| \psi \|_{L^2 (\G)} = \rho \}.
    \end{equation*}
    For $\psi \in \mathcal{M}_\rho$, Lemma~\ref{lemEquivNorms} implies that
    \begin{equation*}
        S_\omega(\psi)
        \le \frac{1}{2}(\lambda_k + \omega)\rho^2
        + \Gamma^{p+1} \rho^{p+1}.
    \end{equation*}
    Since $\lambda_k + \omega < 0$ and $p > 1$,
    for $\rho > 0$ sufficiently small,
    we have $\sup_{\mathcal{M}_\rho} S_\omega < 0$.
    Since $\mathcal{M}_\rho \subset E_-$ is a symmetric set
    of genus $k$ (by \cite[Proposition 7.7]{rabinowitz1986minimax}),
    we have $\mathcal{M}_\rho \in \Sigma_k$.
    Since $\Sigma_k \subseteq \Sigma_j$ for $j \le k$, we obtain
    \begin{equation*}
        c_1 \le c_2 \le \dots \le c_k
        \le \sup_{\mathcal{M}_\rho} S_\omega < 0.
    \end{equation*}

    \medskip
    
    \noindent \textit{Step 2: Convergence.}
    By Proposition~\ref{PSaction},
    the functional $S_\omega$ satisfies
    the Palais-Smale condition at every level $c \in \R$.
    Since the levels $c_j$ are finite,
    standard Lusternik-Schnirelmann theory
    (see e.g.,\,\cite[Theorem 5.7]{rabinowitz1986minimax}
    or \cite[Proposition
    10.8]{Ambrosetti_Malchiodi_2007}) applies.
    Hence, each level $c_j$
    is a critical value of $S_\omega$.
    
    If the levels are distinct ($c_1 < c_2 < \dots < c_k$), we obtain at least $k$ distinct pairs of critical points.
    If some levels coincide,
    say $c_j = c_{j+1} = \dots = c_{j+m} = c$, then the set of critical points $K_c$
    at level $c$ has genus $\gamma(K_c) \ge m+1$
    (see again \cite[Proposition
    10.8]{Ambrosetti_Malchiodi_2007}).
    If $m \ge 1$, then $\gamma(K_c) \ge 2$,
    which implies that $K_c$ is infinite
    (as finite sets have genus 1).
    In all cases, there exist at least $k$
    distinct pairs of solutions.
    We finally note that, since $c_j < 0$
    and $S_\omega(0)=0$, these solutions are non-zero.
    
    \medskip
    
    \noindent \textit{Extension to the complex-valued case.}
    If we relax the assumption on the vertex conditions
    to be Hermitian (allowing complex entries in $\Lambda_v$),
    the operator $\H$ is self-adjoint
    on the complex Hilbert space $L^2(\G, \C)$.
    In this setting, the functional $S_\omega$
    is invariant under the continuous
    $S^1$-action $u \mapsto e^{i\theta}u$ rather
    than the $\Z_2$-action mapping $u$ to $\pm u$.
    Consequently, solutions appear as $S^1$-orbits
    $\mathcal{O}_u = \{ e^{i\theta}u \mid \theta \in [0, 2\pi) \}$.
    
    To prove the existence of $k$ distinct orbits, we replace the Krasnoselskii genus with
    the \textit{geometrical index} (or $S^1$-index)
    introduced by Benci \cite{Benci1981}.
    Let $\mathcal{E}$ be the family of closed,
    $S^1$-invariant subsets of $\DQH \setminus \{0\}$.
    For $A \in \mathcal{E}$,
    the index $i_{S^1}(A)$ is defined as the smallest integer
    $k \ge 1$ such that there exists a continuous
    equivariant map $h : A \to \C^k \setminus \{0\}$.
    (Here, equivariant means $h(e^{i\theta}u) = e^{i\theta}h(u)$
    for all $u \in A$ and $\theta \in [0, 2\pi)$).
    If no such $k$ exists, we define $i_{S^1}(A) = \infty$,
    and by convention $i_{S^1}(\emptyset) = 0$.
    For a detailed presentation of this index and the associated critical point theory,
    we refer to \cite[Chapter 9]{Costa2007} (see in particular Theorem 3.5).
    
    Adapting the previous proof, we note that
    the negative eigenspace $E_-$ corresponding
    to eigenvalues $\lambda_1, \dots, \lambda_k$
    now has \textit{complex} dimension $k$.
    Since the $S^1$-index of the unit sphere
    in $\C^k$ is $k$, the minimax levels $c_j$
    (defined via $i_{S^1}$ instead of $\gamma$)
    are well-defined and strictly negative.
    Equivariant Lusternik-Schnirelmann theory then implies
    the existence of at least $k$ distinct orbits of solutions.
\end{proof}

\subsection{Normalized solutions with a prescribed mass}

We will now turn to the proof of Theorem~\ref{thmMultNormalized}.
This time, the compactness is more delicate.
As we will see in Proposition~\ref{propPSsignMultipliers} below, the sign of the ``approximate Lagrange multipliers''
associated to a Palais-Smale sequence for $E_{\H}$ on the mass constraint will play an important role,
in a similar way to the role it played for energy
ground states (see e.g.\,Lemma~\ref{lemPosMult}).

As for Proposition \ref{prpMuOne}, we will be able to prove compactness when $\mu$ is small and one probably cannot expect to be able to remove the small-mass assumption in general.

Given $\mu > 0$, we define
\begin{equation*}
    \MM_{\mu}
    \coloneqq \bigl\{ \psi \in \DQH \mid
    \| \psi \|_{L^2(\G)} = \mu \bigr\}.
\end{equation*}
We note that $\MM_{\mu}$ is a codimension one
smooth submanifold of $\DQH$.

\begin{proposition}
    \label{propPSsignMultipliers}
    Let $(\psi_n)_n \subseteq \MM_{\mu}$
    be a Palais-Smale sequence for $E_{\H}$
    (constrained on $\MM_{\mu}$)
    at an energy level $c \in \R$. It satisfies
    \begin{equation}
        \label{psi_n_PS}
        \H \psi_n + \omega_n \psi_n
        + |\psi_n|^{p-1}\psi_n = \zeta_n
    \end{equation}
    with $\zeta_n \to 0$ in $\DQH'$
    for a certain sequence $(\omega_n)_n \subseteq \R$
    of Lagrange multipliers\footnote{Note that
    we see $\MM_{\mu}$ as a codimension one real submanifold
    of $\DQH$, so that the $(\omega_n)_n$ are real.}.
    Assume that $(\omega_n)_n$
    is bounded from above and satisfies
    $\liminf_n \omega_n > 0$.
    Then, up to a subsequence,
    $(\psi_n)_n$ converges strongly in $H^1(\G)$
    to some $\psi \in \DQH \cap \MM_{\mu}$
    that is a critical point of $E_{\H}$
    constrained to $\MM_{\mu}$.
\end{proposition} 

\begin{proof}
    \textit{Step 1: Boundedness.} Using Proposition~\ref{prpBddEnergy},
    we deduce that $(\psi_n)_n$ is bounded
    in $H^1(\mathcal{G})$.
    Proposition~\ref{prpStrongCon}
    implies that, up to a subsequence,
    there exists $\psi \in H^1(\mathcal{G})$
    such that $\psi_n \rightharpoonup \psi$ weakly in $H^1 (\G)$.
    By the Rellich-Kondrachov theorem, one has
    $\psi_n \to \psi$ strongly in $L^2_{\operatorname{loc}}(\mathcal{G})$.
    
    \medskip

    \noindent \textit{Step 2: Strong convergence in $L^2(\G)$.}
    Let us fix $\eps > 0$.
    To prove global strong convergence,
    we must establish \textit{tightness}, i.e.,
    that the mass of $\psi_n$ does not escape to infinity
    along the half-lines.
    Let $e \in \mathcal{E}_{\operatorname{inf}}$ identified with $[0, \infty)$.
    We want to show that there exists
    $R_e > 0$ such that for all $n$ large enough,
    we have $\| \psi_{n, e} \|_{L^2([R_e, +\infty))} \le \eps$.
    
    Let $\tilde \theta_R \in C^\infty([0, \infty))$
    be a smooth nonnegative cut-off function
    such that $\tilde \theta_R(x) = 0$
    for $x \le R$ and $\tilde \theta_R(x) = 1$ for $x \ge 2R$.
    Let $\theta_{R} = (\theta_{R, j})_{j \in \mathcal{E}}:
    \G \to \R$ be defined by $\theta_{R, j} \equiv 0$
    for any $j \in \mathcal{E}$, $j \neq e$,
    while $\theta_{R, e} = \tilde \theta_R$.
    Testing equation \eqref{psi_n_PS}
    against $\theta_R \psi_n \in \DQH$ yields
    \begin{multline} \label{eqTestCutoff}
        \langle \H \psi_n, \theta_R \psi_n \rangle_{L^2 (\G)}
        +  \left( \omega_n \int_0^\infty \theta_R (x)
        |\psi_{n, e} (x)|^2 dx
        + \int_0^\infty \theta_R (x)
        |\psi_{n, e} (x)|^{p+1} dx \right) \\
        = \langle \zeta_n, \theta_R \psi_n \rangle_{L^2 (\G)}.
    \end{multline}
    
    Integrating by parts on the first term
    (and noting that boundary terms at $0$ vanish
    due to the support of $\theta_R$), we obtain
    \begin{equation*}
        \langle \H \psi_n, \theta_R \psi_n \rangle_{L^2 (\G)}
        = \int_0^\infty \theta_R (x) |\psi_{n, e}' (x)|^2 dx
        - \frac{1}{2} \int_0^\infty \theta_R''(x) |\psi_{n, e} (x)|^2 dx.
    \end{equation*}
    Substituting this back into \eqref{eqTestCutoff} gives
    \begin{multline*}
        \int_0^\infty \theta_R (x) |\psi_{n,e}' (x)|^2 dx
        + \omega_n \int_0^\infty \theta_R (x) |\psi_{n, e} (x)|^2 dx
        + \int_0^\infty \theta_R (x) |\psi_{n, e} (x)|^{p+1} dx \\
        = \frac{1}{2} \int_0^\infty \theta_R'' (x) |\psi_{n,e} (x)|^2 dx
        + \langle \zeta_n, \theta_R (x) \psi_n \rangle_{L^2 (\G)}.
    \end{multline*}
    
    By assumption, $\liminf_n \omega_n > 0$.
    Hence, there exists $\nu > 0$ such that
    $\omega_n \ge \nu$ for all $n$ large enough. We obtain 
    \begin{equation*}
        \nu \int_0^\infty \theta_R (x) |\psi_{n, e} (x)|^2 dx
        \le \frac{1}{2} \int_R^{2R} |\theta_R'' (x)|
        |\psi_{n, e} (x)|^2 dx
        + \|\zeta_n\|_{(H^1 (\G))'}
        \|\theta_R \psi_n\|_{H^1 (\G)}.
    \end{equation*}
    The second term on the right-hand side vanishes
    as $n \to \infty$ because $\zeta_n \to 0$ in $(H^1 (\G))'$
    and $(\psi_n)_n$ is bounded in $H^1(\G)$.
    For the first term, notice that the integral
    is over the compact interval $[R, 2R]$.
    Since $\psi_n \to \psi$ strongly on compact sets,
    \[
    \limsup_{n \to \infty} \int_R^{2R} |\theta_R'' (x)| |\psi_{n, e} (x)|^2 dx
    = \int_R^{2R} |\theta_R'' (x)| |\psi_{e} (x)|^2 dx.
    \]
    Since $\psi \in L^2(\mathcal{G})$, we can choose $R$ large enough such that this integral
    is smaller than $\frac{\eps^2}{\nu}$.
    Consequently, there exists $R > 0$ such that for all
    $n$ large enough, one has
    \begin{equation*}
        \int_{2R}^\infty |\psi_{n, e} (x)|^2 dx
        \le \eps^2.
    \end{equation*}
    Taking $R_e \coloneqq 2R$ shows tightness on $e$.

    Doing so for all infinite edges of $\G$,
    we deduce that $(\psi_n)_n$ is tight.
    Combined with the local strong convergence on compact sets,
    this implies that $(\psi_n)_n$ converges strongly to
    $\psi$ in $L^2(\mathcal{G})$.

    \medskip
     
    \noindent \textit{Step 3: Strong convergence in $H^1 (\G)$.}
    Testing \eqref{psi_n_PS} with $\psi_n$ yields
    \begin{equation*}
        \langle \H \psi_n, \psi_n \rangle_{L^2 (\G)}
        + \omega_n \|\psi_n\|_{L^2 (\G)}^2
        + \|\psi_n\|_{L^{p+1} (\G)}^{p+1}
        = \langle \zeta_n, \psi_n \rangle_{L^2 (\G)} = o(1).
    \end{equation*}
    We know that $\psi_n \to \psi$ strongly in $L^2 (\G)$ and $L^{p+1} (\G)$ (by interpolation and boundedness in $L^\infty (\G)$).
    Since $(\omega_n)_n$ is bounded, up to a subsequence, there exists $\omega \in \R$ such that $\omega_n \to \omega$.
    Thus, the terms $\omega_n \|\psi_n\|_{L^2 (\G)}^2$ and $\|\psi_n\|_{L^{p+1} (\G)}^{p+1}$ converge to their limits involving $\psi$.
    
    Therefore,
    \begin{equation*}
        \langle \H \psi_n, \psi_n \rangle_{L^2 (\G)}
        \to -\omega \|\psi\|_{L^2 (\G)}^2 - \| \psi \|_{L^{p+1} (\G)}^{p+1}
        = \langle \H \psi, \psi \rangle_{L^2 (\G)}.
    \end{equation*}
    Thus, $Q_{\H}(\psi_n) \to Q_{\H}(\psi)$,
    so that strong convergence in $\DQH$ holds,
    and so does convergence in $H^1 (\G)$
    using Gårding's inequality (Lemma \ref{lemGarding}).
    
    Finally, taking limits in \eqref{psi_n_PS}, we obtain
    \begin{equation*}
        \H \psi + \omega \psi + |\psi|^{p-1}\psi = 0,
    \end{equation*}
    confirming that $\psi$ is a critical point
    constrained on $\MM_{\mu}$.
\end{proof}

We can finally prove Theorem \ref{thmMultNormalized}.

\begin{proof}[Proof of Theorem~\ref{thmMultNormalized}]
    We first assume that the vertex conditions
    are defined by real symmetric matrices
    and search for critical points in the real subspace
    $X_{\R} \coloneqq \DQH \cap H^1(\G, \R)$.
    As for the proof of Theorem~\ref{thmMultAction}, we will use Lusternik-Schnirelmann theory,
    applied to the functional $E_{\H}$
    constrained on $\MM_{\mu} \cap X_{\R}$.
     
    We define $\tilde{\mu}$ by
    \begin{equation}
        \label{defMuTilde}
        \tilde{\mu}
        \coloneqq \left( \frac{-\lambda_k}{2\Gamma^{p+1}}
        \right)^{\frac{1}{p-1}},
    \end{equation}
    where $\Gamma$ is defined in Lemma~\ref{lemEquivNorms}.

    Given an integer $j \ge 1$, we define $\Sigma_j$
    as the family of compact,
    symmetric subsets $A \subset \MM_{\mu}$
    with genus $\gamma(A) \ge j$.
    We define the minimax levels for $1 \le j \le k$ by
    \begin{equation}
        c_j(\mu)
        \coloneqq \inf_{A \in \Sigma_j} \sup_{u \in A} E_{\H}(u).
    \end{equation}
    We will show that, for all $\mu \in (0, \tilde{\mu})$
    and all $1 \le j \le k$, the level $c_j(\mu)$ is negative
    and is a critical level for $E_{\H}$ on $\MM_{\mu}$.

    \medskip
    
    \noindent \textit{Step 1: Energy estimates.}
    Let us obtain quantitative estimates on $c_j$
    in terms of $\mu$.
    \begin{itemize}
        \item \textit{Upper Bound:}
        For $1 \le j \le k$,
        let $V_j \coloneqq \lspan\{\phi_1, \dots, \phi_j\}$.
        The set $A = V_j \cap \MM_{\mu}$ has genus $j$.
        For $\psi \in A$, the quadratic form satisfies
        $\langle \H \psi, \psi \rangle_{L^2 (\G)} \le \lambda_j \mu^2$.
        Using Lemma~\ref{lemEquivNorms}, one has
        \begin{equation} \label{eq:UpBound}
            c_j(\mu) \le \frac{1}{2} \lambda_j \mu^2
            + \frac{\Gamma^{p+1}}{p+1} \mu^{p+1}
            \le \frac{1}{2} \lambda_k \mu^2
            + \frac{\Gamma^{p+1}}{p+1} \mu^{p+1}.
        \end{equation}
        Recalling \eqref{defMuTilde} and the fact that $\lambda_k < 0$,
        one has $c_j(\mu) < 0$ for all $1 \le j \le k$.
        
        \item \textit{Lower Bound:} By the Intersection Lemma
        (see e.g.\,\cite[Proposition
        7.8]{rabinowitz1986minimax}),
        any $A \in \Sigma_j$ intersects
        $V_{j-1}^\perp
        = \lspan\{\phi_1, \dots, \phi_{j-1}\}^\perp$.
        Let $\psi \in A \cap V_{j-1}^\perp$. Then,
        $\langle \H \psi, \psi \rangle_{L^2 (\G)}
        \ge \lambda_j \| \psi \|_{L^2 (\G)}^2$.
        Since the nonlinear term is non-negative, we obtain
        \begin{equation} \label{eq:LowBound}
            c_j(\mu) \ge \frac{1}{2} \lambda_j \mu^2.
        \end{equation}
    \end{itemize}
    Combining \eqref{eq:UpBound} and \eqref{eq:LowBound} yields
    \begin{equation} \label{eq:c_j_est}
        \bigl| c_j(\mu) - \frac{1}{2} \lambda_j \mu^2 \bigr|
        \le \frac{\Gamma^{p+1}}{p+1} \mu^{p+1}.
    \end{equation}

    \noindent \textit{Step 2: The Palais-Smale condition
        at the minimax levels.}
    Let $(\psi_n)_n \subseteq \MM_{\mu}$
    be a Palais-Smale sequence for $E_{\H}$
    (constrained on $\MM_{\mu}$) at level $c_j(\mu)$.
    It satisfies
    \begin{equation*}
        \H \psi_n + \omega_n \psi_n
        + |\psi_n|^{p-1}\psi_n
        = \zeta_n
    \end{equation*}
    for some sequences $(\omega_n)_n \subseteq \R$
    and $(\zeta_n)_n \subseteq (H^1 (\G))'$
    with $\zeta_n \to 0$ in $(H^1 (\G))'$.
    Taking the $L^2$-scalar product with $\psi_n$ yields
    \begin{equation}
        \label{estLambdas}
        \omega_n \mu^2
        + 2 E(\psi_n)
        + \frac{p-1}{p+1} \|\psi_n\|_{p+1}^{p+1}
        = o_n(1),
    \end{equation}
    since $(\psi_n)_n$ is bounded in $L^2 (\G)$.
    Since $(\psi_n)_n$ is a Palais-Smale sequence
    at level $c_j(\mu)$, one has
    $E(\psi_n) = c_j(\mu) + o_n(1)$.
    Recalling \eqref{eq:c_j_est}
    and using Lemma~\ref{lemEquivNorms},
    \eqref{estLambdas} yields
    \begin{equation*}
        \mu^2 |\omega_n + \lambda_j|
        \le |\omega_n \mu^2 + 2E(\psi_n)|
        + \frac{p-1}{p+1} \Gamma^{p+1} \mu^{p+1} + o_n(1)
        \le \Gamma^{p+1} \mu^{p+1} + o_n(1),
    \end{equation*}
    so that
    \begin{equation*}
        |\omega_n + \lambda_j|
        \le \Gamma^{p+1} \mu^{p-1} + o_n(1),
    \end{equation*}
    for all $1 \le j \le k$ and all $n$ large enough.
    By the choice of $\tilde{\mu}$, one has
    $\Gamma^{p+1} \mu^{p-1} \le \frac{-\lambda_k}{2}$.
    Recalling that $\lambda_k$ is negative, this implies
    $\omega_n \ge \frac{-\lambda_k}{2} > 0$
    for $n$ large enough.
    Proposition~\ref{propPSsignMultipliers} thus implies that, up to a subsequence,
    $\psi_n$ converges strongly
    to some $\psi \in \DQH \cap \MM_{\mu}$,
    which is a critical point at level $c_j(\mu)$.
    
    \medskip
    
    \noindent \textit{Step 3: Conclusion.}
    We have shown that the Palais-Smale condition
    holds for all levels $c_j(\mu)$ with $1 \le j \le k$.
    Using the Lusternik-Schnirelmann theory (see
    e.g.\,\cite[Proposition 10.8]{Ambrosetti_Malchiodi_2007}),
    we deduce that all levels $c_1, \dotsc, c_k$
    are critical levels of $E_{\H}$ constrained on $\MM_{\mu}$.
    We deduce the claimed multiplicity result
    as in the action case.
    
    \medskip
    
    \noindent \textit{Extension to the complex-valued case.}
    If the vertex conditions are Hermitian,
    we consider the problem in the complex space $H^1(\G, \C)$.
    The proof extends verbatim by replacing
    the $\Z_2$-genus with the $S^1$-index
    $i_{S^1}$ as in the proof of Theorem~\ref{thmMultAction}.
    Indeed, the Lusternik-Schnirelmann theory holds
    for a general ``index'' and not just the genus
    (see the proof of \cite[Proposition 10.8]{Ambrosetti_Malchiodi_2007}).
    
    Crucially, the energy estimates in \textit{step 1}
    rely only on the modulus $|\psi|$
    and the quadratic form $\langle \H \psi, \psi \rangle_{L^2 (\G)}$,
    both of which remain real-valued and satisfy
    the same bounds in the complex setting.
    Since $V_{j-1}$ has complex dimension $j-1$, the intersection
    of sets with $V_{j-1}^\perp$ remains non-empty.
    Moreover, the Lagrange multipliers remain positive
    since $\mathcal{M}_{\mu}$ is a codimension one \emph{real} manifold.
    Finally, we note that the estimates in \textit{step 2} depend only on norms
    and energy levels, guaranteeing that $\omega_n$
    remains positive, allowing
    the application of Proposition~\ref{propPSsignMultipliers}.
    This yields $k$ distinct $S^1$-orbits of solutions.
\end{proof}

\printbibliography

@misc{BaBoDoTe25NPI,
 author = {Barbera, Daniele and Boni, Filippo and Dovetta, Simone and Tentarelli, Lorenzo},
 title = {Normalized solutions of one-dimensional defocusing {NLS} equations with nonlinear point interactions},
 year = {2025},
 howpublished = {Preprint, {arXiv}:2503.21700 [math.{AP}] (2025)},
 url = {https://arxiv.org/abs/2503.21700},
 arXiv = {arXiv:2503.21700}
}

@article{BoDo22Nl,
 author = {Boni, Filippo and Dovetta, Simone},
 title = {Doubly nonlinear {Schr{\"o}dinger} ground states on metric graphs},
 fjournal = {Nonlinearity},
 journal = {Nonlinearity},
 issn = {0951-7715},
 volume = {35},
 number = {7},
 pages = {3283--3323},
 year = {2022},
 language = {English},
 doi = {10.1088/1361-6544/ac7505},
 zbMATH = {7554805},
 Zbl = {1492.35400}
}

@article{AdBoDo22Nl,
 author = {Adami, Riccardo and Boni, Filippo and Dovetta, Simone},
 title = {Competing nonlinearities in {NLS} equations as source of threshold phenomena on star graphs},
 fjournal = {Journal of Functional Analysis},
 journal = {J. Funct. Anal.},
 issn = {0022-1236},
 volume = {283},
 number = {1},
 pages = {34},
 note = {Id/No 109483},
 year = {2022},
 language = {English},
 doi = {10.1016/j.jfa.2022.109483},
 zbMATH = {7510673},
 Zbl = {1486.35400}
}

@article{DuSh26,
title = {{Ground states for the defocusing nonlinear Schrödinger equation on non-compact metric graphs}},
journal = {Comm. Pure Appl. Anal.},
pages = {},
year = {2026},
issn = {1534-0392},
doi = {10.3934/cpaa.2026029},
url = {https://www.aimsciences.org/article/id/69a54f2d1c9579521d0890e5},
author = {Élio Durand-Simonnet and Boris Shakarov}
}

@article {KiKePe11,
    AUTHOR = {Kirr, E. and Kevrekidis, P. G. and Pelinovsky, D. E.},
     TITLE = {Symmetry-breaking bifurcation in the nonlinear {S}chr\"odinger
              equation with symmetric potentials},
   JOURNAL = {Comm. Math. Phys.},
  FJOURNAL = {Communications in Mathematical Physics},
    VOLUME = {308},
      YEAR = {2011},
    NUMBER = {3},
     PAGES = {795--844},
      ISSN = {0010-3616,1432-0916},
   MRCLASS = {35Q55 (35B20)},
  MRNUMBER = {2855541},
MRREVIEWER = {Hideo\ Takaoka},
       DOI = {10.1007/s00220-011-1361-3},
       URL = {https://doi.org/10.1007/s00220-011-1361-3},
}

@article{AdGaSp25Pot,
 author = {Adami, Riccardo and Gallo, Ivan and Spitzkopf, David},
 title = {Ground states for the {NLS} on non-compact graphs with an attractive potential},
 fjournal = {Networks and Heterogeneous Media},
 journal = {Netw. Heterog. Media},
 issn = {1556-1801},
 volume = {20},
 number = {1},
 pages = {324--344},
 year = {2025},
 language = {English},
 doi = {10.3934/nhm.2025015},
 zbMATH = {8038705},
 Zbl = {1566.35204}
}

@Misc{CoSh24,
    Author = {Le Coz, Stefan and Shakarov, Boris},
    Title = {Ground states on a fractured strip and one dimensional reduction},
    Year = {2024},
    HowPublished = {Preprint, {arXiv}:2411.18187 [math.{AP}]},
    URL = {https://arxiv.org/abs/2411.18187},
    arXiv = {arXiv:2411.18187}
}

@Misc{CoSh25Bk,
 author = {Le Coz, Stefan and Shakarov, Boris},
 title = {Ground {States} for the {Nonlinear} {Schr{\"o}dinger} {Equation} on {Open} {Books} and {Dimensional} {Reduction} to {Metric} {Graphs}},
 year = {2025},
 howpublished = {Preprint, {arXiv}:2512.23286 [math.{AP}]},
 url = {https://arxiv.org/abs/2512.23286},
 arXiv = {arXiv:2512.23286}
}

@article{Ko00,
    title={A semilinear elliptic equation in a thin network-shaped domain},
    author={Kosugi, S.},
    journal={J. Math. Soc. Japan},
    volume={52},
    number={3},
    pages={673--697},
    year={2000},
    publisher={The Mathematical Society of Japan}
}

@Article{KaNoPe22,
    Author = {Kairzhan, Adilbek and Noja, Diego and Pelinovsky, Dmitry E.},
    Title = {Standing waves on quantum graphs},
    FJournal = {Journal of Physics A: Mathematical and Theoretical},
    Journal = {J. Phys. A, Math. Theor.},
    ISSN = {1751-8113},
    Volume = {55},
    Number = {24},
    Pages = {51},
    Note = {Id/No 243001},
    Year = {2022},
    Language = {English},
    DOI = {10.1088/1751-8121/ac6c60},
    zbMATH = {7618040},
    Zbl = {1507.81100}
}

@article{CaFiNo17,
    author = {Cacciapuoti, Claudio and Finco, Domenico and Noja, Diego},
    title = {Ground state and orbital stability for the {NLS} equation on a general starlike graph with potentials},
    fjournal = {Nonlinearity},
    journal = {Nonlinearity},
    issn = {0951-7715},
    volume = {30},
    number = {8},
    pages = {3271--3303},
    year = {2017},
    language = {English},
    doi = {10.1088/1361-6544/aa7cc3},
    zbMATH = {6767907},
    Zbl = {1373.35284}
}

@Article{AdCaFiNo14,
    Author = {Adami, Riccardo and Cacciapuoti, Claudio and Finco, Domenico and Noja, Diego},
    Title = {Constrained energy minimization and orbital stability for the {NLS} equation on a star graph},
    FJournal = {Annales de l'Institut Henri Poincar{\'e}. Analyse Non Lin{\'e}aire},
    Journal = {Ann. Inst. Henri Poincar{\'e}, Anal. Non Lin{\'e}aire},
    ISSN = {0294-1449},
    Volume = {31},
    Number = {6},
    Pages = {1289--1310},
    Year = {2014},
    Language = {English},
    DOI = {10.1016/j.anihpc.2013.09.003},
    zbMATH = {6378484},
    Zbl = {1304.81087}
}

@inproceedings{FuOhOz08,
    title={Nonlinear {S}chr{\"o}dinger equation
    with a point defect},
    author={Fukuizumi, R. and Ohta, M. and Ozawa, T.},
    booktitle={Ann. Inst. H. Poincar\'e C Anal. Non Lin\'eaire},
    volume={25, 5},
    pages={837--845},
    year={2008},
    organization={Elsevier}
}

@book{AlbGeHo05,
 author = {Albeverio, Sergio and Gesztesy, Friedrich and H{\o}egh-Krohn, Raphael and Holden, Helge},
 title = {Solvable models in quantum mechanics.},
 edition = {2nd revised ed.},
 isbn = {0-8218-3624-2},
 year = {2005},
 publisher = {Providence, RI: AMS Chelsea Publishing},
 language = {English},
 zbMATH = {2132167},
 Zbl = {1078.81003}
}

@article{AlBrD95,
 author = {Albeverio, S. and Brze{\'z}niak, Z. and Dabrowski, L.},
 title = {Fundamental solution of the heat and {Schr{\"o}dinger} equations with point interaction},
 fjournal = {Journal of Functional Analysis},
 journal = {J. Funct. Anal.},
 issn = {0022-1236},
 volume = {130},
 number = {1},
 pages = {220--254},
 year = {1995},
 language = {English},
 doi = {10.1006/jfan.1995.1068},
 zbMATH = {764043},
 Zbl = {0822.35002}
}

@Article{KaOh09,
    Author = {Kaminaga, Masahiro and Ohta, Masahito},
    Title = {Stability of standing waves for nonlinear {Schr{\"o}dinger} equation with attractive delta potential and repulsive nonlinearity},
    FJournal = {Saitama Mathematical Journal},
    Journal = {Saitama Math. J.},
    ISSN = {0289-0739},
    Volume = {26},
    Pages = {39--48},
    Year = {2009},
    zbMATH = {5707861},
    Zbl = {1191.35254}
}

@Article{CaLi82,
    Author = {Cazenave, T. and Lions, Pierre-Louis},
    Title = {Orbital stability of standing waves for some nonlinear {Schr{\"o}dinger} equations},
    FJournal = {Communications in Mathematical Physics},
    Journal = {Commun. Math. Phys.},
    ISSN = {0010-3616},
    Volume = {85},
    Pages = {549--561},
    Year = {1982},
    Language = {English},
    DOI = {10.1007/BF01403504},
    zbMATH = {3810068},
    Zbl = {0513.35007}
}

@Book{		  BeKu13,
    author	= {Berkolaiko, Gregory and Kuchment, Peter},
    title		= {Introduction to quantum graphs},
    series	= {Mathematical Surveys and Monographs},
    volume	= 186,
    publisher	= {American Mathematical Society, Providence, RI},
    year		= 2013,
    pages		= {xiv+270},
    isbn		= {978-0-8218-9211-4},
    mrclass	= {81Q35 (05C90 31C20 34B24 34B45 81Q50)},
    mrnumber	= 3013208,
    mrreviewer	= {Delio Mugnolo}
}

@InCollection{KoSc06,
    Author = {Kostrykin, Vadim and Schrader, Robert},
    Title = {Laplacians on metric graphs: eigenvalues, resolvents and semigroups},
    BookTitle = {Quantum graphs and their applications. Proceedings of an AMS-IMS-SIAM joint summer research conference on quantum graphs and their applications, Snowbird, UT, USA, June 19--23, 2005},
    ISBN = {0-8218-3765-6},
    Pages = {201--225},
    Year = {2006},
    Publisher = {Providence, RI: American Mathematical Society (AMS)},
    Language = {English},
    zbMATH = {5082576},
    Zbl = {1122.34066}
}

@book{rabinowitz1986minimax,
    title={Minimax methods in critical point theory with applications to differential equations},
    author={Rabinowitz, Paul H.},
    volume={65},
    year={1986},
    publisher={American Mathematical Soc.},
    series={CBMS Regional Conference Series in Mathematics},
    isbn={978-0-8218-0715-6}
}

@book{Ambrosetti_Malchiodi_2007, place={Cambridge}, series={Cambridge Studies in Advanced Mathematics}, title={Nonlinear Analysis and Semilinear Elliptic Problems}, publisher={Cambridge University Press}, author={Ambrosetti, Antonio and Malchiodi, Andrea}, year={2007}, collection={Cambridge Studies in Advanced Mathematics}}

@article{DDGS,
    author  = {De Coster, Colette and Dovetta, Simone and Galant, Damien and Serra, Enrico},
    title   = {An action approach to nodal and least energy normalized solutions for nonlinear {S}chr{\"o}dinger equations},
    journal = {Ann. Inst. H. Poincar\'e C Anal. Non Lin\'eaire},
    year    = {2025},
    note    = {Online first},
    doi     = {10.4171/AIHPC/160},
    url     = {https://ems.press/journals/aihpc/articles/14353406}
}

@article{dovetta_serra_tilli_2023,
    author  = {Dovetta, Simone and Serra, Enrico and Tilli, Paolo},
    title   = {Action versus energy ground states in nonlinear {S}chr{\"o}dinger equations},
    journal = {Math.Ann.},
    year    = {2023},
    volume  = {385},
    pages   = {1545--1576},
    doi     = {10.1007/s00208-022-02382-z},
    url     = {https://doi.org/10.1007/s00208-022-02382-z}
}

@article{AdCaFiNo16En2,
 author = {Adami, Riccardo and Cacciapuoti, Claudio and Finco, Domenico and Noja, Diego},
 title = {Stable standing waves for a {NLS} on star graphs as local minimizers of the constrained energy},
 fjournal = {Journal of Differential Equations},
 journal = {J. Differ. Equations},
 issn = {0022-0396},
 volume = {260},
 number = {10},
 pages = {7397--7415},
 year = {2016},
 language = {English},
 doi = {10.1016/j.jde.2016.01.029},
 zbMATH = {6561566},
 Zbl = {1336.35316}
}

@Misc{SoVi26GSE,
    author = {Soave, Nicola and Villata, Lorenzo},
    title = {Ground states for the {NLS} equation with combined nonlinearity on periodic metric graphs},
    year = {2026},
    howpublished = {Preprint, {arXiv}:2602.01336 [math.{AP}]},
    url = {https://arxiv.org/abs/2602.01336},
    arXiv = {arXiv:2602.01336}
}

@book{Costa2007,
    title     = {An Invitation to Variational Methods in Differential Equations},
    author    = {Costa, David G.},
    publisher = {Birkhäuser Boston},
    series    = {Birkhäuser Advanced Texts Basler Lehrbücher},
    year      = {2007},
    isbn      = {978-0-8176-4535-9},
    doi       = {10.1007/978-0-8176-4536-6}
}

@article{Benci1981,
    title     = {A geometrical index for the group $S^1$ and some applications to the study of periodic solutions of ordinary differential equations},
    author    = {Benci, Vieri},
    journal   = {Comm. Pure Appl. Math.},
    volume    = {34},
    number    = {4},
    pages     = {393--432},
    year      = {1981},
    publisher = {Wiley},
    doi       = {10.1002/cpa.3160340402}
}

\end{document}